\documentclass[11pt,a4paper]{amsart}

\usepackage[top=3cm,bottom=3cm,left=3.2cm,right=3.2cm,headsep=10pt,letterpaper]{geometry}
\usepackage[utf8]{inputenc}
\usepackage[english]{babel}
\usepackage{bm} 
\usepackage{graphicx}
\usepackage{subcaption}
\usepackage{mathtools}
\mathtoolsset{showonlyrefs} 
\usepackage{amsmath}
\usepackage{amsthm,amstext,amssymb,mathtools,esint,tabularx}
\usepackage{caption}

\newcommand{\xx}{\mathbf{x}}

\usepackage{tikz,pgfplots,pgfplotstable} 
\pgfplotsset{compat=newest}
\usetikzlibrary{calc,3d,arrows,patterns,decorations}
\tikzset{>=stealth'}
\usetikzlibrary{positioning} 
\usepackage[abs]{overpic}
\usepackage[outline]{contour}
\contourlength{0.6pt}
\usetikzlibrary{decorations.text}
\tikzset{
  text deco/.style={postaction={decorate, decoration={text along path,#1}}}}
\usepackage{yfonts}

\usepackage{ifpdf} \ifpdf  \fi 

\numberwithin{equation}{section}

\newtheorem{thm}{Theorem}
\newtheorem{lem}[thm]{Lemma}
\newtheorem{rem}[thm]{Remark}

\newtheorem{deftion}[thm]{Definition}

\usepackage{amsaddr}
\begin{document}

\title{Energy-corrected FEM and explicit time-stepping for parabolic problems}
\author{P.~Swierczynski \and B~Wohlmuth}
\address{Institute for Numerical Mathematics,\\Technische Universit\"{a}t M\"{u}nchen,\\85748 Garching b. M\"{u}nchen, Germany}
\email{piotr.swierczynski@ma.tum.de, wohlmuth@ma.tum.de}

\keywords{Corner singularities, second-order parabolic equations, energy-corrected FEM}

\begin{abstract}
The presence of corners in the computational domain, in general, reduces the regularity of solutions of parabolic problems and diminishes the convergence properties of the finite element approximation introducing a so-called "pollution effect". Standard remedies based on mesh refinement around the singular corner result in very restrictive stability requirements on the time-step size when explicit time integration is applied. In this article, we introduce and analyse the energy-corrected finite element method for parabolic problems, which works on quasi-uniform meshes, and, based on it, create fast explicit time discretisation. We illustrate these results with extensive numerical investigations not only confirming the theoretical results but also showing the flexibility of the method, which can be applied in the presence of multiple singular corners and a three-dimensional setting. We also propose a fast explicit time-stepping scheme based on a piecewise cubic energy-corrected discretisation in space completed with mass-lumping techniques and numerically verify its efficiency.
\end{abstract}
\maketitle

\section{Introduction}\label{Sec:intro}

Numerical approximations of parabolic problems have been extensively studied in multiple settings and are of great interest due to the many technical applications, in which they appear~\cite{Cengel2014,Lienhard2017,Price1970}. Standard approximation methods involve finite differences~\cite{Lewy1928,Thomee1990}, but to allow for computations on more complicated domains, the theory of finite element methods has been developed, among many others, in~\cite{Douglas1970, Fix1972,Wheeler1973}. For a more exhaustive discussion of the standard finite element approximations of parabolic problems and an extensive list of references, we refer the reader to~\cite{Thomee2006}.

The presence of corners in the computational domain negatively influences regularity properties of the solutions of elliptic~\cite{Grisvard1985} and parabolic problems~\cite{Banasiak1991,Grisvard1985}, due to the appearance of certain known singular functions. For corners with angles $\Theta > \pi$, in general, $H^2$ regularity in space cannot be guaranteed. This introduces the so-called pollution effect, diminishing the convergence order of the standard finite element schemes both for elliptic~\cite{Blum1990,Blum1982} and parabolic problems~\cite{Chatzipantelidis2006}. Standard methods for improving the approximation properties of the finite element method in the elliptic setting include refinement~\cite{Babuska1970,Babuska2001} and grading~\cite{Apel2001,Apel1996,Schatz1979} of the mesh around the singular corner. These results can also be extended to parabolic problems on non-convex polygonal domains~\cite{Chatzipantelidis2006}. However, due to the very small mesh size in the vicinity of the re-entrant corner, the CFL condition~\cite{Lewy1928}, which guarantees the stability of an explicit time integration scheme, becomes very restrictive. This, in turn, means that explicit time-stepping schemes cannot be efficiently used, as they would require prohibitively small time-steps.

In this article, we follow a different approach, based on the so-called energy corrected finite element method. The idea was originally proposed for finite difference schemes in~\cite{Ru89,RZ86,Gietl1978}. Recently, it has been extended to cover finite element methods for the Poisson equation in~\cite{Egger2014,Huber2015,Ruede2014,Swierczynski2018}, the Stokes equation in~\cite{John2015} and optimal control problems in~\cite{John2018}. The energy-correction method eliminates the pollution effect from the finite element approximation by a scaling of a fixed number of entries in the stiffness matrix. As opposed to adaptivity and grading, it can be successfully applied on quasi-uniform meshes, and hence, does not suffer from a too severe CFL condition. This permits the construction of fast explicit time-stepping schemes combined with the energy-corrected finite element in space.

This article is structured as follows: in Section~\ref{Sec:parabolic} we discuss the regularity properties of the parabolic equations and briefly review the main results concerning the energy-correction method. In Section~\ref{Sec:FEM_parabolic}, we generalize the energy-corrected finite element to parabolic problems and provide a complete error analysis. We illustrate the analysis with numerical investigations in Section~\ref{sec:NumericalResults}.
In Section~\ref{sec:Extensions}, we present several potential extensions of the scheme. We show that it can be applied to the advection-diffusion problem with a moderate advection term.  Furthermore, we introduce higher-order energy-corrected discretisation combined with mass-lumping technique and a post-processing approach improving convergence properties of the scheme also in the vicinity of the singular corner. We complete the discretisation with an explicit Runge-Kutta time-stepping and show that the energy-corrected scheme exhibits superior performance compared with other commonly used discretisation methods. Finally, we present a potential application of the method, showing the flexibility of the energy-corrected finite element, which can be applied to problems with multiple re-entrant corners, and in a three-dimensional setting.
\section{Parabolic problem}\label{Sec:parabolic}
Let $\Omega\subset \mathbb{R}^2$ be a bounded, non-convex, polygonal domain, i.e., a domain containing a re-entrant corner of size $\pi < \Theta < 2 \pi$. For the sake of simplicity we shall assume that this domain contains only one such corner. Note however that the analysis presented here also applies in a more general setting of domains with an arbitrary number of re-entrant corners.\\
Consider a standard heat equation defined on~$\Omega$ in a time interval~$[0, T]$ with $T > 0$
\begin{align}\label{eq:HeatEquation}
u_t - \Delta u &= f \quad \text{in } \Omega\times (0, T),\\
u &= 0 \quad \text{on } \partial \Omega\times [0, T],\\
u &= u_0 \quad \text{in } \Omega \text{ at } t = 0 .
\end{align}
We define a corresponding weak solution as a continuous function $u: [0, T] \rightarrow H^1_0(\Omega)$ such that $u(0) = u_0 \; a.e.$  and for all $v \in H^1_0 (\Omega) \text{ and } a.e.\; 0 < t \leq T$
\begin{align}\label{eq:HeatEquationWeakSolution}
\langle u_t (t), v \rangle + a(u(t), v)  &= \langle f, v \rangle
\end{align}
where $a(u(t), v) \coloneqq  \big\langle \nabla u(t), \nabla v \big\rangle$.

The unique solution to this problem exists and its regularity on smooth and convex domains has been studied extensively, see~\cite{Evans2010}. Having smooth initial conditions $u_0\in C^\infty (\overline{\Omega})$ and a smooth source term $f\in C^\infty (\overline{\Omega} \times [0, T])$ the solution of Equation~\eqref{eq:HeatEquation} is also smooth, see~\cite[Chapter 7, Theorem 7]{Evans2010}. However, this does not hold anymore, if polygonal domains are concerned.

\subsection{Definitions and auxiliary results}
We present our analysis in weighted Sobolev spaces as they constitute a convenient framework for the description of elliptic and parabolic problems defined on domains with re-entrant corners. Let $r(\xx)$ be the Euclidean distance of $\xx\in\mathbb{R}^2$ from the re-entrant corner. For any $\beta \in \mathbb{R}$ we define a weighted Sobolev space as the following vector space
\begin{align*}
H^m_\beta (\Omega) \coloneqq \big\lbrace u \text{ - measurable } : \; r^{\beta + |\mu |  - m}D^\mu u \in L^2(\Omega), \; \; 0 \leq |\mu | \leq m \big\rbrace,
\end{align*}
where $\mu$ is a multiindex with non-negative entries. Moreover, we equip this space with a norm
\begin{align*}
\| u \|_{m ,\beta} \coloneqq \bigg( \sum_{|\mu | \leq m} \big\| r^{\beta + |\mu |  - m}D^\mu u \big\|^2_{L^2(\Omega)}  \bigg)^{1/2}.
\end{align*}
For convenience we denote $L^2_\beta(\Omega) \coloneqq H^0_\beta (\Omega)$. We also use the notational conventions $\| u\|_{0, \beta} = \|u \|_{\beta}$ and $\| u\|_{L^2(\Omega)} = \| u \|_0$. Similarly as standard Sobolev spaces, weighted spaces form a natural hierarchy, see~\cite{Kufner1985}
\begin{align}\label{eq:Kufner}
H^{m + l}_{\beta + l} (\Omega) \hookrightarrow H^m_{\beta} (\Omega ), \quad \text{for any } l \in \mathbb{Z}_+.
\end{align}
In particular for $m = 0$ and $l=1$ we have
\begin{align}\label{eq:Kufner_v2}
\| v \|_{-\alpha} \leq c_\alpha \| \nabla v \|_{0}, \quad \text{for any } v \in H^1_0(\Omega).
\end{align}
for some~$c_\alpha > 0$ depending only on~$\alpha$ and the domain~$\Omega$.

\subsection{Regularity results}\label{sec:EllipticRegularity}
In order to present regularity results for the parabolic problem on non-convex domains, we would like to first summarise regularity properties of a related elliptic problem. We choose this approach, as elliptic and parabolic problems exhibit the same type of singular behaviour in the vicinity of a re-entrant corner.

Consider
\begin{equation}\label{eq:EllipticModelProblem}
-\Delta w  = f \quad \text{in } \Omega,\quad w = 0 \quad \text{on } \partial\Omega.
\end{equation}
Moreover, let $1 - \alpha < \pi/ \Theta$.  It is well known, see, e.g., \cite{Kondratjev1967}, that in the case of domains with re-entrant corners, the $H^2(\Omega)$ regularity of the solution~$w$ usually cannot be obtained regardless of the regularity of the forcing term~$f$. However, the solution $w$ can be split into regular and singular parts, namely
\begin{align}\label{eq:SingularRegularSplit}
w =  \sum_{0 < n < (1 + \alpha) \Theta/ \pi} k_n s_n +  W.
\end{align}
Here, $W\in H^2_{-\alpha}(\Omega)$ denotes the smooth remainder, and $s_n$ are the singular functions defined as
\begin{align}\label{eq:SingularFunctions}
s_n(r, \phi) = \eta(r) r^{n\pi/\Theta}\sin \Big(\frac{n\pi}{\Theta}\phi\Big),
\end{align}
where $(r, \phi)$ are the polar coordinates in the vicinity of a re-entrant corner corresponding to the angle~$\Theta$. Moreover, $\eta$ is a smooth cut-off function equal to $1$ in an arbitrary neighbourhood with a fixed distance from the re-entrant corner and equal to $0$ far from it. Notice that $w \in H^2_{\alpha}(\Omega)$. Furthermore, a precise formula for the stress-intensity factors $k_n$ is known~\cite[Chapter~8]{Grisvard1985}
\begin{align}\label{eq:StressIntensityFactor}
k_n = -\frac{1}{n \pi}\int_\Omega f s_{-n} + u \Delta s_{-n}.
\end{align}


The following regularity result for the parabolic system~\eqref{eq:HeatEquation} was proposed in~\cite{Banasiak1991}.
\begin{thm}\label{thm:Regularity}
Let $u(t)$ be a weak solution of the heat equation~\eqref{eq:HeatEquationWeakSolution} on a polygonal domain $\Omega\subset \mathbb{R}^2$ with a re-entrant corner of size~$\Theta$. Let also $f\in C^\sigma \big([0, T], L^2(\Omega)\big)$, $\sigma > 0$. Then there exists $U \in C \big([ 0, T]; H^2 (\Omega) \big)$ and $k_1 (t) \in C \big( [0, T) \big) \cap C^1 \big( ( 0, T ) \big)$ such that
\begin{align*}
u(t, \mathbf{x}) = U(t, \mathbf{x}) +  k_1(t) s_1(\mathbf{x}).
\end{align*}
Moreover, $u(t, \cdot) \in H^2_{\alpha}(\Omega)$ for all $t > 0$.
\end{thm}

The regularity of parabolic equations on domains with conical points was further studied in~\cite[Theorem~3.2]{Nguyen2008} with the results presented in the framework of weighted Sobolev spaces. The following theorem is an important consequence of the analysis presented there
\begin{thm}\label{thm:RegularityU}
Suppose that the assumptions of Theorem~\ref{thm:Regularity} are satisfied. Furthermore, let $f\in L^2\big( 0,T; H^4_{-\alpha}(\Omega) \big)$, $\frac{\mathrm{d}f}{\mathrm{d}t} \in L^2 \big( 0, T; H^2_{-\alpha} (\Omega) \big)$, $\frac{\mathrm{d}^2f}{\mathrm{d}t^2} \in L^2 \big( 0, T; L^2_{-\alpha} (\Omega) \big)$, and $u_0 \in H^3_{-\alpha}(\Omega)$. Assume also that the standard compatibility condition $f(0) + \Delta u_0 \in H^1(\Omega)$ is satisfied. Then,
\begin{align*}
\max_{0\leq t\leq T} \| \Delta u \|_{-\alpha} < \infty, \quad \text{and } \int_0^T \| \Delta u_t \|_{-\alpha}^2 < \infty, \quad \text{and } \int_0^T \big\| u_{tt} \big\|_\alpha^2 < \infty.
\end{align*}
\end{thm}

We would like to stress out that the same singular functions arise in the solution of both elliptic and parabolic problems on non-convex polygonal domains. Note also that the results of Theorem~\ref{thm:Regularity} and Theorem~\ref{thm:RegularityU} can be easily extended to multiple singular functions.

\subsection{Finite element discretisation}\label{sec:FEM}
Suppose that $\mathcal{T}_H$ is an admissible and shape-regular coarse initial triangulation of the domain~$\overline{\Omega}$. We introduce the space of piecewise linear and globally continuous functions by
\begin{align*}
S_h \coloneqq \{ v \in \mathcal{C}(\overline{\Omega}) : v|_T \in P_1(T),\; \text{for all } T \in \mathcal{T}_h \},
\end{align*}
and define the conforming finite element space by
\begin{align}\label{eq:FiniteElementSpace}
V_h \coloneqq S_h \cap H_0^1(\Omega) \subset H_0^1(\Omega).
\end{align}

The following inverse inequality is a fundamental property of the discrete spaces and follows from the equivalence of norms in the finite dimentional spaces, see~\cite[Section~6.8]{Braess_2007} for more details.

\begin{lem}[Global inverse inequality]\label{lem:StandardInverseInequality}
Let $0\leq l \leq m \leq k$. There exists a constant~$c>0$ such that
\begin{align*}
| v_h |_{H^m(\Omega)} \leq c_i h_{\min}^{l - m} | v_h |_{H^l (\Omega)}, \quad \text{for all } v_h \in S_h,
\end{align*}
where $h_{\min} = \min_{T\in \mathcal{T}_h} h_T$. In particular, for uniformly refined meshes, we have
\begin{align*}
| v_h |_{H^m(\Omega)} \leq c_i h^{l - m} | v_h |_{H^l (\Omega)}, \quad \text{for all } v_h \in S_h,
\end{align*}
where $c_1 h \leq h_{\min} \leq c_2 h$ for some $c_1, c_2 > 0$.
\end{lem}

From now on, we assume that $\mathcal{T}_h$ be a mesh obtained by uniform refinement of~$\mathcal{T}_H$. 

We define the semi-discretisation of the problem~\eqref{eq:HeatEquationWeakSolution} in space as finding $u_h, u_{h,t}[0, T] \rightarrow V_h$ such that
\begin{align}\label{eq:FEMSemiDiscrete}
\langle u_{h, t}(t), v_h \rangle + a(u_h (t), v_h)  &= \langle f, v_h \rangle \quad \text{for all } v_h \in V_h, \text{ and } t\in (0, T]\\
u_h(0) &= \mathcal{P}_h u_0.
\end{align}
The operator $\mathcal{P}_h : L^2 (\Omega) \rightarrow S_h$ is linear and will be precisely specified later. For now, we assume that $\mathcal{P}_h$ is the $L^2$-projection. It is well known~\cite{Brezzi2013,Thomee2006} that then the finite element approximation in a sufficiently regular setting exhibits optimal second-order convergence in the $L^2$-norm and first-order convergence in the energy $H^1$-norm, namely for $a.e.\; t\in [0, T]$
\begin{align*}
\big\| \nabla ( u - u_h ) \big \|_{0} \lesssim  h, \quad \| u - u_h \|_{0} \lesssim h^2.
\end{align*}
The situation is significantly different when re-entrant corners of maximum angle $\Theta > \pi$ are present in the computational domain.

Let us now consider the model elliptic boundary value problem~\eqref{eq:EllipticModelProblem}. Remember, $\Omega$ is a polygonal domain with a re-entrant corner of angle $\pi < \Theta < 2\pi$. 
In case of the standard piecewise linear finite element approximation of the model problem~\eqref{eq:EllipticModelProblem}, we find $w_h \in V_h$ such that
\begin{align}\label{eq:ECFiniteElementModel}
a(w_h, v_h) = \langle f, v_h \rangle \quad \text{for all } v_h\in V_h.
\end{align}
Due to the reduced regularity of the solution of~\eqref{eq:EllipticModelProblem}, as summarised in Section~\ref{sec:EllipticRegularity}, the convergence order of the finite element approximation~\eqref{eq:ECFiniteElementModel} is also not optimal, when measured in the standard and weighted $L^2(\Omega)$-norms. This behaviour is known as the so-called \textit{pollution effect}, see, e.g.,~\cite{ Blum1990,Blum1982, Brenner2008, Ciarlet1991, Strang1988}.
\begin{thm}[Pollution effect]\label{thm:PollutionEffect}
Let $w$ be the solution of~\eqref{eq:EllipticModelProblem} with $f\in L^2_{-\alpha}(\Omega)$ for some $1 - \alpha < \pi / \Theta$. Further, assume that $k_1 \neq 0$, then
\begin{align*}
\| w - w_h \|_{ \alpha} \gtrsim \| \nabla (w - w_h)\|_{0}^2 \gtrsim h^{2\pi/\Theta}.
\end{align*}
\end{thm}
The proof of Theorem~\ref{thm:PollutionEffect} can be found in \cite{Egger2014}. Notice that the suboptimal approximation order is also obtained far from the re-entrant corner, so even in the case of elliptic equations, standard piecewise polynomial finite element approximation yields suboptimal convergence order.\\
This translates directly to parabolic problems, since in the presence of non-convex corners in the polygonal domain, the following convergence rates can be observed~\cite{Chatzipantelidis2006}
\begin{align}\label{eq:ConvStFEMParabolic}
\big\| \nabla ( u - u_h ) \big \|_{0} \lesssim h^{\pi/ \Theta}, \quad \| u - u_h \|_{0} \lesssim h^{2\pi/ \Theta}.
\end{align}
These rates can be improved using suitable mesh-grading techniques, so that the optimal convergence in space is regained~\cite{Thomee2006}. However, the corresponding CFL condition for explicit time-stepping schemes, meaning that the time-step needs to be scaled like a square of the size of the smallest element~\cite{Lewy1928}, gets very prohibitive and makes the use of explicit time-stepping schemes less attractive.
\subsection{Energy-corrected finite element for elliptic equations}\label{sec:EnergyCorrection}
Here, we give a brief overview of the energy-correction techniques used for improving the convergence order in the finite element approximations of elliptic problems on polygonal domains. The idea was originally proposed for finite difference schemes in \cite{Ru89,RZ86,Gietl1978} and has been further developed recently in the finite element setting in~\cite{Egger2014,Ruede2014,Swierczynski2018} and is based on a local modification of the bilinear form governing the problem. It was later extended to piecewise polynomial approximation spaces in~\cite{Horger2016HO}. One of the advantages of this method, which we will exploit later, is the possibility of using quasi-uniform meshes.
The pollution effect in Theorem~\ref{thm:PollutionEffect} is a result of an insufficient approximation of the energy $| a(u, u) - a(u_h, u_h) |$ by standard finite element techniques on uniform meshes. In order to remove this effect in the finite element approximation~\eqref{eq:ECFiniteElementModel}, we introduce a modification of the bilinear form $a(\cdot, \cdot)$, which shall mitigate the stiffness of the problem in the vicinity of the singularity.
The modified finite element approximation of~\eqref{eq:EllipticModelProblem} reads then: find $w_h^{\text{m}} \in V_h$ such that
\begin{align}\label{eq:ECFiniteElementModified}
a_h(w_h^{\text{m}}, v_h) = \langle f, v_h \rangle \quad \text{for all } v_h \in V_h,
\end{align}
where the bilinear form is defined as $a_h(w, v) \coloneqq a(w, v) - c_h(w, v)$. We assume that $a_h(\cdot, \cdot)$ is bilinear, continuous and elliptic, namely there exist $c_b, c_b > 0$ such that for all $v_h, w_h \in V_h$
\begin{align}\label{eq:ECBoundedCoercive}
a_h(v_h, w_h) \leq c_b \| \nabla v_h \|_{H^1_0(\Omega)} \| \nabla w_h \|_{H^1_0(\Omega)}, \quad \text{and }\quad  a_h(v_h, v_h) \geq c_c \| \nabla v_h \|^2_{H^1_0(\Omega)}.
\end{align}
Furthermore, we assume that $c_h(\cdot, \cdot)$ is symmetric. One possible choice of the modification is
\begin{align}\label{eq:ECModification}
c_h (w, v)  = \gamma\int_{\omega_h} \nabla w \cdot \nabla v,
\end{align}
where $\omega_h$ is a one element patch around the re-entrant corner and $0 < \gamma < 1/2$. Due to the choice of the modification $c_h(\cdot, \cdot)$, we preserve the sparsity structure of the stiffness matrix, as only a small, fixed number of its entries needs to be suitably scaled. An additional assumption of symmetry of the nodal patch $\omega_h$ is necessary for angles $\Theta \geq 3\pi /2$.

Consequently, taking into account the energy-corrected formulation introduced in~\eqref{eq:ECFiniteElementModified}, we define a modified Ritz projection $R^{\text{m}}_h : H^1_0(\Omega) \rightarrow V_h$ by
\begin{align}\label{eq:ModifiedRitzProjection}
a_h(R_h^{\text{m}} w, v_h) = a(w, v_h) \quad \text{for all } v_h \in V_h.
\end{align}

Let $s_{1,h}^{\text{m}} \in V_h$ denote the modified Ritz projection of the singular function $s_1$. The following theorem, providing sufficient conditions for the optimal convergence of the energy corrected method, was proposed in \cite{Egger2014}.

\begin{thm}\label{thm:EnergyCorrectioOptConvergence}
Let $1 - \alpha < \pi/\Theta$ and $f\in L^2_{-\alpha}(\Omega)$. Let the modification $c_h(\cdot, \cdot)$ be defined as above and satisfy
\begin{align*}
a(s_1 - s^{\mathrm{m}}_{1, h}, s_1 - s^{\mathrm{m}}_{1, h}) - c_h(s_{1,h}^{\mathrm{m}}, s_{1,h}^{\mathrm{m}}) =\mathcal{O}( h^2).
\end{align*}
Then, for the energy-corrected finite element solution we obtain the following optimal error estimates
\begin{align*}
\|w - w^{\mathrm{m}}_h\|_{\alpha} \lesssim h^2 \| f\|_{-\alpha}, \quad \|w - w^{\mathrm{m}}_h\|_{1,  \alpha} \lesssim h \| f \|_{- \alpha}.
\end{align*}
\end{thm}

Asymptotically, as $h \rightarrow 0$, there exists a unique, independent of~$h$ optimal parameter $\gamma^*$ on the correction patch $\omega_h$.
Note that the correction parameter depends on the number and shape of the elements $T$ of the correction patch $\omega_h$ and on the angle $\Theta$ of the re-entrant corner. Several effective procedures for finding it, based on nested Newton strategies, were proposed in \cite{Ruede2014}.


Since weighted norms are not as commonly used as standard $L_2$-norms, we would like to obtain similar results to Theorem~\ref{thm:EnergyCorrectioOptConvergence} for the latter. This can be done by exploiting some prior knowledge about the form of the singularity in the solution.\\
According to the decomposition~\eqref{eq:SingularRegularSplit} into singular and regular parts of the exact solution of~\eqref{eq:EllipticModelProblem}, and by the linearity of the problem, we can represent the energy-corrected finite element approximation of $w$ as
\begin{align}
w_h^{\text{m}} = k_1 s_{1, h}^{\text{m}} + R_h^{\text{m}} W,
\end{align}
and the stress-intensity factor~\eqref{eq:StressIntensityFactor} can be efficiently approximated by
\begin{align}\label{eq:StressIntensityFactorDiscrete}
k_1^h = -\frac{1}{\pi}\int_\Omega f s_{-1} + w_h^{\text{m}} \Delta s_{-1}.
\end{align}
Then, due to Theorem~\ref{thm:EnergyCorrectioOptConvergence}, we immediately obtain
\begin{align}\label{eq:StressIntensityFactorConvergence}
| k_1 - k_1^h | \lesssim h^2 \|f \|_{-\alpha}.
\end{align}
We define the post-processed approximation by
\begin{align}\label{eq:PostProcessing}
\widetilde{w}_h^{\text{m}} := w_h^{\text{m}} - k_1^h s_{1, h}^{\text{m}} + k_1^h s_1 = k_1^h s_1 + R_h^{\text{m}} W + \big( k_1 - k_1^h \big) s_{1, h}^{\text{m}}. 
\end{align}
\begin{thm}
The post-processed solution $\widetilde{w}_h^{\text{m}}$ defined in~\eqref{eq:PostProcessing} converges with an optimal rate in standard norms, namely
\begin{align}
\| w - \widetilde{w}_h^{\mathrm{m}} \|_{0} \lesssim h^2 \| f \|_{-\alpha}, \quad \text{and } \quad \| \nabla (w - \widetilde{w}_h^{\mathrm{m}}) \|_{0} \lesssim h \| f \|_{ -\alpha}
\end{align}
\end{thm}
For a more detailed analysis of the post-processing we refer the reader to~\cite[Section~2]{Horger2016HO}.
\section{Energy-corrected finite elements for parabolic problem}\label{Sec:FEM_parabolic}
In this section, we investigate the energy-corrected finite element approximation of the parabolic problem~\eqref{eq:HeatEquation}. We begin by analysing the semi-discretisation, where the time variable is continuous. The reasoning is then further extended to a fully discrete case with the explicit Euler time-stepping.

\subsection{Energy-corrected semi-discrete scheme}
We define a modified semi-discrete finite element approximation as
\begin{align}\label{eq:FEMParabolic}
\begin{cases}
\langle u_{h, t}^{\text{m}}, v_h \rangle + a_h(u_h^{\text{m}}, v_h) = \langle f, v_h\rangle,&\quad \text{for all } v_h \in V_h\\
u^{\text{m}}_h(0) = R_h^{\text{m}} u_0, &\quad\text{ }
\end{cases}
\end{align}
where $a_h(\cdot, \cdot)$ is an energy-corrected bilinear form introduced in Section~\ref{sec:EnergyCorrection}. To impose the initial conditions we also use the modified Ritz projection~\eqref{eq:ModifiedRitzProjection} in place of the operator $\mathcal{P}_h$ in~\eqref{eq:FEMSemiDiscrete}.
\begin{lem}[Stability of the semi-discrete scheme]\label{lem:DiscreteStability}
Let $f \in L^2_\alpha(\Omega)$ for some $\alpha < 1$. The semi-discrete solution of problem~\eqref{eq:FEMParabolic} satisfies for some~$C^*>0$
\begin{align*}
\| u_h^{\mathrm{m}} \|^2_\alpha \leq \| R_h^{\mathrm{m}} u_0 \|_\alpha^2 + C^* \int_0^T \| f \|_\alpha^2 \; \mathrm{d}t.
\end{align*}
\end{lem}
\begin{proof}
An application of the ellipticity of the bilinear form~$a_h(\cdot, \cdot)$, see~\eqref{eq:ECBoundedCoercive}, the Cauchy-Schwarz and Young inequalities, together with the choice $v_h = u_h^{\text{m}}$ in Equation~\eqref{eq:FEMParabolic}, yields for any~$\epsilon>0$

\begin{align*}
\frac{1}{2}\frac{\mathrm{d}}{\mathrm{d}t} \| u_h^{\text{m}}\|_0^2 + c_c\| \nabla u_h^{\text{m}} \|_0^2 &\leq \langle f, u_h^{\text{m}}\rangle \\ &\leq \|f \|_{\alpha}\| u_h^{\text{m}} \|_{-\alpha} 
\leq \frac{1}{2\epsilon}\| f \|^2_{\alpha} + \frac{\epsilon }{2} \| u_h^{\text{m}} \|_{-\alpha}^2.
\end{align*}

Using the embedding~\eqref{eq:Kufner_v2} $H^1(\Omega) \hookrightarrow L^2_{-\alpha}(\Omega)$ for $1 - \alpha > 0$, we get
\begin{align*}
\frac{1}{2}\frac{\mathrm{d}}{\mathrm{d}t} \| u_h^{\text{m}}\|_0^2 + c_c\| \nabla u_h^{\text{m}} \|_0^2 \leq \frac{1}{2\epsilon}\| f \|^2_{\alpha} + \frac{\epsilon c_\alpha}{2} \| \nabla u_h^{\text{m}} \|_0^2,
\end{align*}
where $c_\alpha$ is a constant coming from Inequality~\eqref{eq:Kufner_v2}. Choosing $\epsilon \leq 2 c_c/c^2_\alpha$ we arrive at
\begin{align*}
\frac{1}{2}\frac{\mathrm{d}}{\mathrm{d}t} \| u_h^{\text{m}}\|_0^2 \leq C^*\| f \|^2_{\alpha},
\end{align*}
where $C^* \geq c^2_\alpha/4c_c$. Integrating both sides over the time interval~$[0,T]$ completes the proof of the lemma.
\end{proof}
\begin{rem}
In the case of the standard choice of the modified bilinear form~\eqref{eq:ECModification} in the energy-corrected scheme we have $c_c = 1$ and the stability constant $C^*$ above can be reduced to $C^* = c^2_\alpha/4$. The embedding constant $c_\alpha$ from~\eqref{eq:Kufner_v2} is an equivalent of the Poincar\'{e}-Friedrichs constant in weighted spaces and depends only on the domain~$\Omega$.
\end{rem}
%
\begin{thm}\label{thm:MainResult}
Suppose that functions $u_0$ and $f$ satisfy the regularity requirements stated in Theorem~\ref{thm:Regularity} and Theorem~\ref{thm:RegularityU} and let $0 < 1 - \alpha < \pi/\Theta$. The energy corrected semi-discretisation~\eqref{eq:FEMParabolic} of Problem~\eqref{eq:HeatEquation} yields optimal convergence rate in the weighted $L^2$-norm, namely for some~$c>0$ independent of~$u$
\begin{align}\label{eq:MainResult}
\max_{0\leq t\leq T}\| u - u_h^{\mathrm{m}}\|_{\alpha} \leq c h^2 \bigg( \max_{0\leq t\leq T}\| \Delta u (t) \|^2_{-\alpha} + \int_0^T \| \Delta u_t(t) \|_{-\alpha}^2 \; \mathrm{d}t \bigg)^{1/2}.
\end{align}
\end{thm}
\begin{proof}
We proceed in a standard manner by splitting the discretisation error into two independent parts

\begin{align}\label{eq:ErrorSplit}
u(t) - u^{\text{m}}_h(t) &= \Big( u(t) - R^{\text{m}}_h u(t) \Big) + \Big( R^{\text{m}}_h u(t) - u_h^{\text{m}}(t)\Big)  =: \rho + \eta,
\end{align}
where $R_h^{\text{m}}$ denotes the energy-corrected Ritz projection defined in Equation~\eqref{eq:ModifiedRitzProjection}. Hence, due to Theorem~\ref{thm:EnergyCorrectioOptConvergence}

\begin{align}\label{eqConvOfRho1}
\| \rho \|_{\alpha}  &= \| u(t) - R^{\text{m}}_h u(t) \|_{\alpha} \leq ch^2 \| \Delta u(t) \|_{-\alpha}
\end{align}
and
\begin{align}\label{eqConvOfRho2}
\| \rho_t \|_{\alpha}  &= \| u_t(t) - R^{\text{m}}_h u_t(t) \|_{\alpha} \leq ch^2 \| \Delta u_t(t) \|_{-\alpha}.
\end{align}

Using the definition of the modified Ritz projection~\eqref{eq:ModifiedRitzProjection}, definition of the continuous solution~\eqref{eq:HeatEquationWeakSolution} and the energy-corrected discretisation~\eqref{eq:FEMParabolic} we arrive at
\begin{align*}
\big< \eta_t, v_h \big> + a_h (\eta, v_h) = \big< - \rho_t, v_h \big>.
\end{align*}

Finally, due to Lemma~\ref{lem:DiscreteStability}, we get

\begin{align*}
\| R^{\text{m}}_h u(t) - u_h^{\text{m}}(t) \|_\alpha^2 &= \| \eta (t) \|_\alpha^2  \leq \| \eta (0) \|_0^2 + C^* \int_0^t \| \rho_t \|_\alpha^2 \; \mathrm{d}t.
\end{align*}

Note that due to Theorem~\ref{thm:RegularityU}, the right-hand side of the inequality above is well-defined.  Moreover, the discrete initial conditions were chosen in a way that~$\eta(0) = 0$.

Combining this with Equation~\eqref{eq:ErrorSplit} gives

\begin{align*}
\max_{0\leq t\leq T}\| u - u_h^{\text{m}}\|^2_{\alpha} \leq \max_{0\leq t\leq T} \Big(\| \rho \|^2_\alpha + \| \eta \|^2_\alpha \Big) \leq  \bigg(\max_{0\leq t\leq T}\| \rho \|^2_\alpha + C^*\int_0^T \| \rho_t \|_\alpha^2 \; \mathrm{d}t \bigg).
\end{align*}

Finally, application of the results stated in~\eqref{eqConvOfRho1} and~\eqref{eqConvOfRho2} completes the proof.

\end{proof}

The right-hand side of Equation~\eqref{eq:MainResult} is finite, see Theorem~\ref{thm:RegularityU}. The above theorem shows that the application of the energy-corrected finite element scheme to the parabolic equations results in the optimal accuracy of the scheme, when compared to the interpolation error.

\subsection{Energy-corrected fully discrete scheme}

Now, we move to the fully discrete setting, where also the temporal dimention is discretised. We consider only explicit Euler time-stepping, which later will serve as a foundation for building fast numerical schemes. The extension to a more general case of $\theta$-scheme in time is straightforward.

The fully discrete energy-corrected finite element approximation of the model problem~\eqref{eq:HeatEquationWeakSolution} reads as follows: Find $U_h^{\text{m}, n} \in V_h$ for $0 \leq n \leq N$ such that
\begin{align}\label{eq:EETimeStepping}
\Big< \frac{U_h^{\text{m}, n + 1} - U_h^{\text{m}, n}}{\Delta t}, v_h \Big> + a_h(U_h^{\text{m}, n}, v_h) = \big<f (t_n), v_h \big>, \quad \text{for all } v_h \in V_h.
\end{align}
The initial condition, as before, is imposed using the modified Ritz projection~\eqref{eq:ModifiedRitzProjection}
\begin{align*}
U^{\text{m},0}_h = R_h^\text{m} u_0.
\end{align*}

In order to investigate the stability of the explicit Euler scheme, we need to introduce the so-called CFL (Courant--Friedrichs--Lewy) condition, which holds for all admissible triangulations.

\begin{deftion}\label{def:CFLCondition}
Consider the explicit Euler time-stepping introduced above. We define the CFL condition as
\begin{align}\label{eq:CFLCondition}
h^{-2}_{\min} \Delta t \leq c_s,
\end{align}
where $h_{\min} = \min_{T \in \mathcal{T}_h} h$ and $c_s$ is a constant independent of the triangulation and the time-step~$\Delta t$.
\end{deftion}

This condition was first introduced in~\cite{Lewy1928} in the context of finite difference methods. The extension concerning the finite element methods for time-dependent problems can be found in~\cite{Baker1977,Brenner1982}. Note that in the case of uniform meshes $h_{\min}$ can be replaced with the mesh size~$h$ in~\eqref{eq:CFLCondition}.

We begin the convergence analysis of the scheme by showing an auxiliary result bouding the finite difference in the formulation~\eqref{eq:EETimeStepping}. Let $c_i > 0$ denote the constant appearing in the inverse inequality from Lemma~\ref{lem:StandardInverseInequality}, when $l=0$ and $m=1$. Furthermore, let $c_\alpha >0$ and $c_b>0$ be respectively the constants in inequalities~\eqref{eq:Kufner_v2} and~\eqref{eq:ECBoundedCoercive}.

\begin{lem}\label{lem:DiscreteAux1}
Suppose that~$f\in C \big( 0, T; L^2_\alpha (\Omega) \big)$ for some~$0 \leq \alpha < 1$. Then for all~$0\leq n \leq N-1$
\begin{align*}
\Big\|  \frac{U^{\mathrm{m}, n+1}_h - U^{\mathrm{m}, n}_h}{\Delta t} \Big\|_0 \leq c_i h^{-1} \Big( c_\alpha \| f \|_\alpha + c_b \| \nabla U^{\mathrm{m}, n}_h \|_0 \Big).
\end{align*}
\end{lem}

\begin{proof}
Let us set $v_h =  \frac{U^{\text{m}, n+1}_h - U^{\text{m}, n}_h}{\Delta t} $ in~\eqref{eq:EETimeStepping}. Then, applying the Cauchy-Schwarz inequality and using the boundedness of the bilinear form~$a_h(\cdot, \cdot)$, we get

\begin{align*}
\Big\|  \frac{U^{\text{m}, n+1}_h - U^{\text{m}, n}_h}{\Delta t} \Big\|^2_0 &= \Big< f(t_n) , \frac{U^{\text{m}, n+1}_h - U^{\text{m}, n}_h}{\Delta t} \Big> - a_h \Big(U^{\text{m}, n}_h, \frac{U^{\text{m}, n+1}_h - U^{\text{m}, n}_h}{\Delta t}\Big)\\
&\leq \| f \|_\alpha \Big\|  \frac{U^{\text{m}, n+1}_h - U^{\text{m}, n}_h}{\Delta t} \Big\|_{-\alpha} + c_b \| \nabla U^{\text{m},n}_h \|_0 \Big\|\nabla  \frac{U^{\text{m}, n+1}_h - U^{\text{m}, n}_h}{\Delta t} \Big\|_0.
\end{align*}

Due to~\eqref{eq:Kufner_v2}, we obtain

\begin{align*}
\Big\|  \frac{U^{\text{m}, n+1}_h - U^{\text{m}, n}_h}{\Delta t} \Big\|^2_0 \leq \Big( c_\alpha\| f \|_\alpha   + c_b \| \nabla U^{\text{m},n}_h \|_0\Big) \Big\| \nabla \frac{U^{\text{m}, n+1}_h - U^{\text{m}, n}_h}{\Delta t} \Big\|_0.
\end{align*}

Finally, application of the inverse inequality from Lemma~\ref{lem:StandardInverseInequality} yields the desired result.
\end{proof}

Now, we can state the stability result in weighted Sobolev spaces, which will prove crucial for showing the error estimates for the fully discrete scheme. Similarly as in the case of the standard norms, the explicit time-stepping scheme is stable only under an additional assumption that the CFL condition~\eqref{eq:CFLCondition} is satisfied. We provide the precise value of the stability constant.

\begin{thm}[Stability of the fully discrete scheme]\label{thm:ECParabolicStability}
Suppose that for some $0 \leq \alpha < 1$ we have $f\in C \big( 0, T; L^2_\alpha (\Omega) \big)$ and let $0 < \epsilon< 1/2$, $0< \delta < \frac{c_c}{c_\alpha^2}$. Suppose also that the CFL condition proposed in Definition~\ref{def:CFLCondition} is satisfied with the constant $c_s = 2\frac{c_c -  c_\alpha^2 \delta}{c_i^2 c_b^2(1 + \epsilon)}$. Then for some~$c_{\epsilon, \delta}>0$ independent of~$h$ and~$\Delta t$, we have

\begin{align*}
\| U^{\mathrm{m}, n}_h \|^2_0 \leq \| U^{\mathrm{m}, 0}_h \|^2_0 + c_{\epsilon, \delta} \Delta t \sum_{k=0}^{n-1} \| f(t_n) \|^2_\alpha.
\end{align*}
\end{thm}
\begin{proof}
We set $v_h = U^{\text{m}, n}_h$ in~\eqref{eq:EETimeStepping}. Notice that
\begin{align*}
U^{\text{m},n}_h = \frac{U^{\text{m}, n+1}_h + U^{\text{m},n}_h}{2} - \frac{\Delta t}{2} \frac{U^{\text{m}, n+1}_h - U^{\text{m}, n}_h}{\Delta t}.
\end{align*}
Hence
\begin{align}\label{eq:StabilityAux1}
\frac{\| U^{\text{m},n+1}_h\|^2_0 - \|U^{\text{m},n}_h\|^2_0}{2\Delta t} + a_h (U^{\text{m},n}_h, U^{\text{m},n}_h) = \big< f(t_n), U^{\text{m},n}_h\big> + \frac{\Delta t}{2}\Big\|  \frac{U^{\text{m}, n+1}_h - U^{\text{m}, n}_h}{\Delta t} \Big\|^2_0.
\end{align}

Note that for any numbers $\epsilon, a,b >0$ we have
\begin{align*}
(a + b )^2 \leq \big(1 + \frac{1}{\epsilon} \big) a^2 + (1 + \epsilon ) b^2.
\end{align*}
Therefore, Lemma~\ref{lem:DiscreteAux1} gives us

\begin{align}\label{eq:StabilityAux2}
\Big\|  \frac{U^{\text{m}, n+1}_h - U^{\text{m}, n}_h}{\Delta t} \Big\|^2_0 \leq c_i^2 h^{-2} c_\alpha^2\big(1 + \frac{1}{\epsilon} \big) \|f(t_n) \|^2_\alpha + c_i^2 h^{-2} c_b^2(1 + \epsilon ) \| \nabla U^{\text{m},n}_h \|_0^2.
\end{align}

Furthermore, for any $\delta > 0$ we get due to the Cauchy-Schwarz inequality and Equation~\eqref{eq:Kufner_v2}
\begin{align}\label{eq:StabilityAux3}
\big< f(t_n), U^{\text{m},n}_h \big>  \leq \frac{1}{\delta}\|f(t_n) \|^2_\alpha + c_\alpha^2 \delta \| \nabla U^{\text{m},n}_h \|_0^2.
\end{align}

Using the coercivity of the bilinear form~$a_h(\cdot, \cdot)$~\eqref{eq:ECBoundedCoercive} and applying~\eqref{eq:StabilityAux1}--\eqref{eq:StabilityAux3}, we obtain

\begin{align*}
\frac{\| U^{\text{m},n+1}_h\|^2_0 - \|U^{\text{m},n}_h\|^2_0}{2\Delta t} + c_c \| \nabla U^{\text{m},n}_h \|_0^2 \leq \Big( \frac{1}{\delta} &+ \frac{1}{2}\big(1 + \frac{1}{\epsilon}\big) c_i^2 c_\alpha^2 h^{-2}\Delta t \Big) \| f(t_n) \|_\alpha^2 \\
&+ \Big( c_\alpha^2 \delta +\frac{1}{2} c_i^2 c_b^2( 1 + \epsilon ) h^{-2} \Delta t \Big)  \| \nabla U^{\text{m},n}_h \|_0^2.
\end{align*}

The CFL condition~\eqref{eq:CFLCondition} states that
\begin{align*}
h^{-2}\Delta t \leq c_s = 2\frac{c_c - c_\alpha^2 \delta}{c_i^2 c_b^2 (1 + \epsilon)}
\end{align*}
and therefore
\begin{align*}
\frac{\| U^{\text{m},n+1}_h\|^2_0 - \|U^{\text{m},n}_h\|^2_0}{2\Delta t} \leq \Big( \frac{1}{\delta} + \frac{1}{2}\big(1 + \frac{1}{\epsilon}\big) c_i^2 c_b^2 h^{-2}\Delta t \Big) \| f(t_n) \|_\alpha^2
\end{align*}
Setting
\begin{align*}
c_{\epsilon, \delta} =  2 \Big(\frac{1}{\delta} + \frac{1}{2}\big(1 + \frac{1}{\epsilon}\big) c_i^2 c_b^2 c_s\Big)
\end{align*}
and applying induction we finally obtain
\begin{align*}
\| U^{\text{m}, n}_h \|_0^2 \leq \| U^{\text{m}, 0}_h \|_0^2 + c_{\epsilon, \delta} \Delta t \sum_{k=0}^{n-1} \| f(t_n) \|^2_\alpha.
\end{align*}

\end{proof}
Upon the right choice of the values~$\epsilon, \delta$, we see that any $c_s < 2\frac{c_c}{c_i^2 c_b^2}$ is a feasible stability constant. Note however that when $\delta \rightarrow 0$ or $\epsilon \rightarrow 0$, then $c_{\epsilon, \delta} \rightarrow \infty$.

Finally, we are in a position to state the convergence result for the fully discrete scheme.

\begin{thm}\label{thm:FDMainResultParabolic}
Suppose that functions $u_0$ and $f$ satisfy the regularity requirements stated in Theorem~\ref{thm:MainResult} and let $1 - \lambda_1 < \alpha < 1$. Suppose also that the CFL condition stated in Definition~\ref{def:CFLCondition} holds with the constant $c_s = 2\frac{c_c -  c_\alpha^2 \delta}{c_i^2 c_b^2(1 + \epsilon)}$ for some $0 < \epsilon< 1/2$, $0< \delta < \frac{c_c}{c_\alpha^2}$. Then, the following error estimate for the energy-corrected discretisation $U^{\text{m},n}_h$, see~\eqref{eq:EETimeStepping}, of Problem~\eqref{eq:HeatEquation} holds for some~$c>0$ independent of~$u$
\begin{align}\label{eq:FDMainResultParabolic}
\max_{0 \leq n \leq N}\| u(t_n) &- U_h^{\mathrm{m},n}\|_{\alpha} \\
&\leq c ( h^2  + \Delta t)  \bigg( \max_{0\leq t\leq T}\| \Delta u (t) \|^2_{-\alpha} + \int_0^T \| \Delta u_t(t) \|_{-\alpha}^2 \; \mathrm{d}t  +   \int_0^T \| u_{tt}\|^2_\alpha \; \mathrm{d}t\bigg)^{1/2}.\nonumber
\end{align}
\end{thm}

\begin{proof}
Similarly as in the proof of Theorem~\ref{thm:MainResult}, we begin the proof by splitting the error into two components.
\begin{align}\label{eq:FDErrorSplitting}
 u(t_n) - U_h^{\text{m},n} = \Big( u(t_n) - R^\text{m}_h u(t_n) \Big) + \Big(R^\text{m}_h u(t_n) - U_h^{\text{m},n}\Big)  =: \rho^n + \eta^n,
\end{align}
where $R_h^m$ denotes the energy-corrected Ritz projection defined in Equation~\eqref{eq:ModifiedRitzProjection}. Due to Theorem~\ref{thm:EnergyCorrectioOptConvergence}

\begin{align}\label{eq:FDConvOfRho1}
\| \rho^n \|_{\alpha}  &= \| u(t_n) - R^\text{m}_h u(t_n) \|_{\alpha} \leq ch^2 \| \Delta u(t_n) \|_{-\alpha}
\end{align}

We focus now on estimating the remaining~$\eta^n$ component of the error. Due to the definition of the energy-corrected Ritz projection and the problem formulation~\eqref{eq:HeatEquationWeakSolution}, we have
\begin{align*}
\big<u_t (t_n), v_h \big> + a_h (R^\text{m}_h u(t_n), v_h) = \big< f(t_n) , v_h\big>, \quad \text{for all } v_h \in V_h.
\end{align*}
Therefore, equation~\eqref{eq:EETimeStepping} yields for all~$v_h \in V_h$

\begin{align*}
a_h ( \eta^n, v_h) &= \Big< \frac{U^{\text{m}, n+1}_h - U^{\text{m},n}_h}{\Delta t} - u_t(t_n), v_h\Big> \\
&= -\Big< \frac{R^\text{m}_h u(t_{n+1}) - R^\text{m}_h u(t_n)}{\Delta t} - \frac{U^{\text{m}, n+1}_h - U^{\text{m}, n}_h}{\Delta t}, v_h \Big> \\
&\hspace{3cm}-\Big< u_t(t_n) - \frac{R^\text{m}_h u(t_{n+1}) - R^\text{m}_h u(t_n)}{\Delta t} , v_h \Big>\\
&= - \Big< \frac{\eta^{n+1} - \eta^n}{\Delta t}, v_h \Big> + \Big< \frac{u(t_{n+1}) - u(t_n)}{\Delta t} - u_t (t_n) , v_h \Big> - \Big< \frac{\rho^{n+1} - \rho^n}{\Delta t}, v_h \Big>.
\end{align*}

Thus, we can write
\begin{align*}
\Big< \frac{\eta^{n+1} - \eta^n}{\Delta t}, v_h \Big>  + a_h ( \eta^n, v_h) = \big< \psi_1^n + \psi_2^n , v_h\big>, \quad \text{for all }v_h \in V_h,
\end{align*}
where
\begin{align*}
\psi_1^n = \frac{u(t_{n+1}) - u(t_n)}{\Delta t} - u_t (t_n), \quad \text{and } \quad \psi_2^n = \frac{\rho^{n+1} - \rho^n}{\Delta t}.
\end{align*}

Thanks to the stability estimate stated in Theorem~\ref{thm:ECParabolicStability} we obtain

\begin{align}\label{eq:FDConvOfRho3}
\| \eta^n \|^2 \leq \| \eta^0 \| + 2\Delta t c_{\epsilon, \delta}\Big( \sum_{k = 0}^{n-1} \| \psi_1^n \|^2_\alpha + \sum_{k = 0}^{n-1} \| \psi_2^n \|^2_\alpha \Big).
\end{align}

We now estimate $\psi_1^n$ and $\psi_2^n$ separately. Note that
\begin{align*}
\psi_1^n = \frac{u(t_{n+1}) - u(t_n)}{\Delta t} - u_t (t_n) = -\frac{1}{\Delta t} \int_{t^n}^{t^{n+1}} (t_{n+1} - t ) u_{tt}\; \mathrm{d}t,
\end{align*}
and hence
\begin{align}\label{eq:FDConvOfRho4}
\| \psi_1^n \|_\alpha \leq  \sqrt{\Delta t} \Big( \int_{t^n}^{t^{n+1}} \| u_{tt}\|^2_\alpha \; \mathrm{d}t\Big)^{1/2}.
\end{align}

Further, due to the linearity of the modified Ritz projection, we have
\begin{align*}
\psi^n_2 = \frac{u(t_{n+1}) - u(t_n)}{\Delta t} - R^\text{m}_h \frac{u(t_{n+1}) - u(t_n)}{\Delta t},
\end{align*}
and thus, see Theorem~\ref{thm:EnergyCorrectioOptConvergence}, we also get
\begin{align*}
\| \psi_2^n \|_\alpha &\leq ch^2 \Big\| \Delta \Big( \frac{u(t_{n+1}) - u(t_n)}{\Delta t} \Big) \Big\|_{-\alpha} = ch^2 \Big\| \frac{1}{\Delta t} \int_{t_n}^{t_{n+1}} \Delta u_t \; \mathrm{d} t \Big\|_{-\alpha}.
\end{align*}
Further, applying Cauchy-Schwarz inequality, we obtain
\begin{align}\label{eq:FDConvOfRho5}
\| \psi_2^n \|_\alpha \leq c \frac{h^2}{\sqrt{\Delta t}} \Big( \int_{t_n}^{t_{n+1}} \|\Delta u_t \|^2_{-\alpha}\; \mathrm{d} t \Big)^{1/2}.
\end{align}

Since the initial conditions in the discretisation are imposed using the modified Ritz projection, see~\eqref{eq:EETimeStepping}, we automatically have $\eta^0 = 0$. Using this and combining~\eqref{eq:FDConvOfRho3} with \eqref{eq:FDConvOfRho4} and~\eqref{eq:FDConvOfRho5}, we arrive at

\begin{align*}
\| \eta^n \|^2 \leq c (\Delta t)^2  \int_0^T \| u_{tt}\|^2_\alpha \; \mathrm{d}t  + c h^4 \int_0^T \|\Delta u_t \|^2_{-\alpha}\; \mathrm{d} t
\end{align*}
Finally, combining this result with~\eqref{eq:FDConvOfRho1}, and applying to the error splitting~\eqref{eq:FDErrorSplitting} we get

\begin{align*}
\max_{0\leq n \leq N}\| u(t_n) &- U_h^{\text{m},n}\|_\alpha^2 \\
&\leq c \big( h^4 + (\Delta t)^2  \big) \Big( \max_{0\leq n \leq N} \| \Delta u(t_n) \|^2_{-\alpha}  + \int_0^T \|\Delta u_t \|^2_{-\alpha}\; \mathrm{d} t  +   \int_0^T \| u_{tt}\|^2_\alpha \; \mathrm{d}t\Big)
\end{align*}

Hence, the proof is completed upon taking the square root of both sides of the inequality. The boundedness of the right-hand side is ensured by Theorem~\ref{thm:RegularityU}.

\end{proof}

As opposed to the mesh grading strategy, the energy-correction works on uniform meshes with less restrictive CFL stability condition~\eqref{eq:CFLCondition}. We shall exploit this fact further in the next section, when creating fast time-stepping schemes.
\section{Numerical results}\label{sec:NumericalResults}
In this section, we propose and numerically investigate a fast solver for parabolic problems based on energy-corrected finite element~\eqref{eq:FEMParabolic}. We show that, as opposed to the algorithms involving mesh grading and adaptivity, explicit time-stepping schemes are a feasible choice in the proposed setting.

We discretize Eq.~\eqref{eq:FEMParabolic} using the Explicit Euler time-stepping scheme. Let us divide the time interval $[0, T]$ into $N \in \mathbb{Z}_+$ time steps of equal lengths $\Delta t$, so $t_n = n \Delta t$. The fully discrete approximation of the model problem~\eqref{eq:HeatEquationWeakSolution} reads as follows: Find $u_h^{m, n} \in V_h$ for $0 \leq n \leq N$ such that
\begin{align}\label{eq:EETimeStepping}
\Big( \frac{u_h^{m, n + 1} - u_h^{m, n}}{\Delta t}, v_h \Big) + a_h(u_h^{m, n}, v_h) = \big(f (t_n), v_h \big), \quad \text{for all } v_h \in V_h.
\end{align}

Let $\big\lbrace \varphi_i \big\rbrace_{i = 1}^K$ be a standard nodal finite element basis. The discrete systems~\eqref{eq:EETimeStepping} can be rewritten in a matrix-vextor formulation as
\begin{align}\label{eq:P1withEEDiscrete}
\mathbf{U}_{n + 1}^m = \mathbf{U}_n^m + \Delta t \mathbf{M}^{-1}\Big[ \mathbf{F}^n - \mathbf{S}^m \mathbf{U}_n^m   \Big],
\end{align}
where $\mathbf{M} =\Big[ \big(  \varphi_i, \varphi_j \big)\Big]_{i,j=1}^K$ and $\mathbf{S}^m =\Big[ a_h(  \varphi_i, \varphi_j \big)\Big]_{i,j=1}^K$ denote the standard mass and energy-corrected stiffness matrices respectively, and $\mathbf{F}^n = \Big[ \big(  f(t_n, \cdot), \varphi_i \big)\Big]_{i=1}^K$. The choice of a nodal, vertex-based quadrature rule for assembling the mass matrix leads to a lumped diagonal matrix~$\widetilde{\mathbf{M}}$, which can be used in place of~$\mathbf{M}$, see~\cite{Bartels_book} for more details. This results in a fast time-stepping scheme, where at each time step, multiplication by a diagonal matrix $\widetilde{\mathbf{M}}$ and a sparse matrix~$\mathbf{S}^m$ needs to be performed.

Stability of the Explicit Euler scheme is guaranteed by the CFL condition~\cite{Thomee2006}, meaning that the size of the time step needs to scale like the square of the mesh size, i.e. $\Delta t \sim h^2$. This is very prohibitive when mesh grading or adaptivity is concerned. However, this is not an issue in the case of the energy-corrected FEM, which works on uniform meshes. Then, balancing the error of order $\mathcal{O}(\Delta t)$ coming from the time-stepping discretization with $\mathcal{O}(h^2)$ order of error measured in the weighted $L^2$-norm, see Thm.~\ref{thm:MainResult}, exactly the same relationship needs to be kept.

In order to improve the convergence of the scheme at a fixed point in time $T$, we complete the algorithm with a post-processing strategy, following the post-processing strategy as in~\eqref{eq:PostProcessing}. As stated in Eq.~\eqref{eq:StressIntensityFactor}, the stress-intensity factor, see Thm.~\ref{thm:Regularity}, can be computed by
\begin{align*}
k_1(T) = -\frac{1}{\pi}\int_\Omega\big( f(T)  - u_t(T) \big)s_{-1} + u \Delta s_{-1}.
\end{align*}
We define its discrete approximation using~\eqref{eq:StressIntensityFactorDiscrete} as
\begin{align*}
k_1^h(T) = -\frac{1}{\pi}\int_\Omega\bigg( f(T) - \frac{u^{m, N}_h - u^{m, N-1}_h}{\Delta t} \bigg)s_{-1} + u_h^m(T) \Delta s_{-1}.
\end{align*}
This leads to the post-processed solution of the form
\begin{align}\label{eq:PostProcessedParabolic}
\widetilde{u}_h^m(T) = u_h^m(T) + k_1^h(T)\big( s_1 - s_{1, h}^m  \big).
\end{align}
Note that the additional cost of performing the post-processing is equal to the cost of solving one additional elliptic equation and evaluating one integral.

In Table~\ref{tab:ConvergenceRates}, we summarise the errors and the convergence rates of the proposed scheme. We choose the L-shape domain $\Omega = (-1, 1)^2 \setminus \big( [0, 1]\times [-1, 0]\big)$ with the largest interior angle of size $\Theta = 3\pi/2$. We also choose a known exact solution $u= \sin(t) s_1 + \sin (2t) s_2 - \sin (3t) s_3$, being a linear combination of singular functions~\eqref{eq:SingularFunctions} with smooth time-dependent coefficients.

The parameter~$\gamma$ in the modification~\eqref{eq:ECModification} is computed using a version of the Newton algorithm described in~\cite{Ruede2014} and in the numerical experiments we choose the weight $\alpha = 1 - \pi/\Theta$. This choice of the weight induces a slightly stronger norm than assumed in Sec.~\ref{Sec:FEM_parabolic} but the optimal convergence order of the energy-corrected scheme can be observed regardless of this. 

We consider uniform refinement of the initial mesh and together with refining the mesh, we also divide the time-step $\Delta t$ by $4$, initially set to be equal to~$0.1$. For the purpose of comparison, in the first two columns of Table~\ref{tab:ConvergenceRates}, we summarise the results obtained using the uncorrected scheme. The suboptimal convergence rates in the sense of the best approximation error, when measured both in standard and weighted $L^2(\Omega)$-norms, are in line with the results in~\eqref{eq:ConvStFEMParabolic}. For the energy-corrected scheme~\eqref{eq:EETimeStepping}, we see that no pollution in the $L^2(\Omega)$-norm appears. Moreover, second-order convergence in the weighted norm means that the error is relatively large only in the vicinity of the re-entrant corner, so the pollution effect from Thm.~\ref{thm:PollutionEffect} has been removed. Finally, the post-processing approach yields second-order convergence in the standard $L^2(\Omega)$-norm. Numerical tests confirm the theoretical results of Theorem~\ref{thm:MainResult}.
\begin{table}[htb]
\centering
\resizebox{\textwidth}{!}{%
\begin{tabular}{c|cc|cc|cc|cc|cc}
L & $\| u - u_h\|_0$ & rate & $\| u - u_h \|_{0, \alpha}$ & rate &$\| u - u^m_h\|_0$ & rate & $\| u - u^m_h \|_{0, \alpha}$ & rate & $\| u - \tilde{u}^m_h \|_{0}$ & rate \\
\hline
1 &   9.9471e-02 &   &          7.7723e-02 &  &          1.0172e-01 &   &          7.9832e-02 &  &             6.8808e-02 &   \\
2 &     3.3940e-02 & 1.55 &          2.3263e-02 & 1.74 &       3.2843e-02 & 1.63  &       2.2848e-02    & 1.80 &        2.2432e-02 & 1.62   \\
3 &       1.2351e-02 & 1.46  &          7.6646e-03 & 1.60 &     9.5573e-03 & 1.78 &       5.6370e-03    & 2.02  &            5.7524e-03 & 1.96 \\
4 &        4.6633e-03 & 1.41  &          2.7192e-03 & 1.50  &      2.7736e-03 & 1.78  &     1.3492e-03      & 2.06    &          1.3562e-03 & 2.08  \\
5 &        1.7942e-03 & 1.38  &          1.0130e-03 & 1.42 &      8.2665e-04 & 1.74 &      3.2694e-03     & 2.04       &       3.1236e-04 & 2.12  \\
6 &        6.9729e-04 & 1.36  &          3.8825e-04 & 1.38  &      2.5229e-04 & 1.71  &          8.0793e-05 & 2.02 &                   7.1894e-05 & 2.12  \\
\end{tabular}
}
\caption{Summary of convergence rates obtained using two different approximations of the heat equation on the L-shape domain}\label{tab:ConvergenceRates}
\end{table}
\section{Extensions}\label{sec:Extensions}
In this section, we present extensions of the methods introduced above. We show that the energy-corrected finite element can be applied to domains with multiple re-entrant corners, also in the presence of a moderate advection in the problem~\eqref{eq:HeatEquation}. Furthermore, we show a possible extension to the higher-order piecewise polynomial finite elements and propose a fast explicit time-stepping scheme based on cubic elements combined with mass-lumping techniques. Finally, in order to show the flexibility of the energy-correction method, we present numerical experiments involving multiple re-entrant corners in three dimensions.

\subsection{Advection-diffusion equation}
In~\cite{Swierczynski2018}, pointwise error estimates for the energy-corrected finite element method for the elliptic problems on polygonal domains were studied. It was shown that the energy-corrected discretisation~\eqref{eq:ECFiniteElementModified} of~\eqref{eq:EllipticModelProblem} yields, up to a logarithmic factor, optimal convergence in the sense of the optimal approximation property. In the following, we show that the improved pointwise convergence of the energy-corrected scheme can be also expected in the case of parabolic problems.

We consider the following advection-diffusion problem
\begin{align}\label{eq:AdvectionDiffusion}
u_t + b\cdot \nabla u - \Delta u &= f \quad \text{in } \Omega\times (0, T),\\
u &= 0 \quad \text{on } \partial \Omega\times [0, T],\\
u &= u_0 \quad \text{in } \Omega \text{ at } t = 0 .
\end{align}
In the numerical example, we consider $T = 1$, $u_0 = 0$ and $b = (1, 1)$, so the problem is equipped with moderate advection. The computational domain $\Omega$, together with its triangulation, is presented in Figure~\ref{fig:Triangle} and consists of a rectangle $(0, 4)\times (0,3)$ with a right, isosceles triangle cut out. There are three re-entrant corners in the domain~$\Omega$, two of sizes~$7\pi/4$ and one of size~$3\pi/2$. We use a computational grid with one-element patches around the singular corners consisting of the identical isosceles triangles. For the right-hand side we choose~$f = \sin(\pi t)\Big( (x - 2)^2 + (y - 3/2)^2 \Big)^{-1}$, which has a singularity in the middle of the cut-out triangle. For the time discretisation, we choose Explicit Euler time-stepping as described in Section~\ref{sec:NumericalResults} with an initial step-size $\Delta t = 0.02$, which is small enough to guarantee the stability of the scheme.

We investigate the behaviour of a quantity of interest $\text{QoI} = \| u_h (T) \|_{L^\infty(\Omega)}$ for standard finite element method and the energy-corrected finite element method on $5$ consecutive refinement levels. We use the same modification of the bilinear form as before, namely~\eqref{eq:ECModification}. The results of the simulations are summarised in the plot on the right-hand side of Figure~\ref{fig:Triangle}. For completeness, we also include the extrapolated approximation $\| u (T)\|_\infty^{ex}$ of the real value in the plot. The estimated order of convergence of $| \| u(T)\|_\infty^{ex} - \| u_h (T) \|_\infty |$ is equal to $1.85$ and $1.55$ in the case of the energy-corrected scheme and standard finite element respectively. The energy-corrected finite element can be successfully applied also in the cases of several different re-entrant corners in the domain and the presence of a moderate advection in the problem.

\begin{figure}
\centering
\begin{minipage}{0.46\textwidth}
\huge
\resizebox{\textwidth}{!}{
\includegraphics{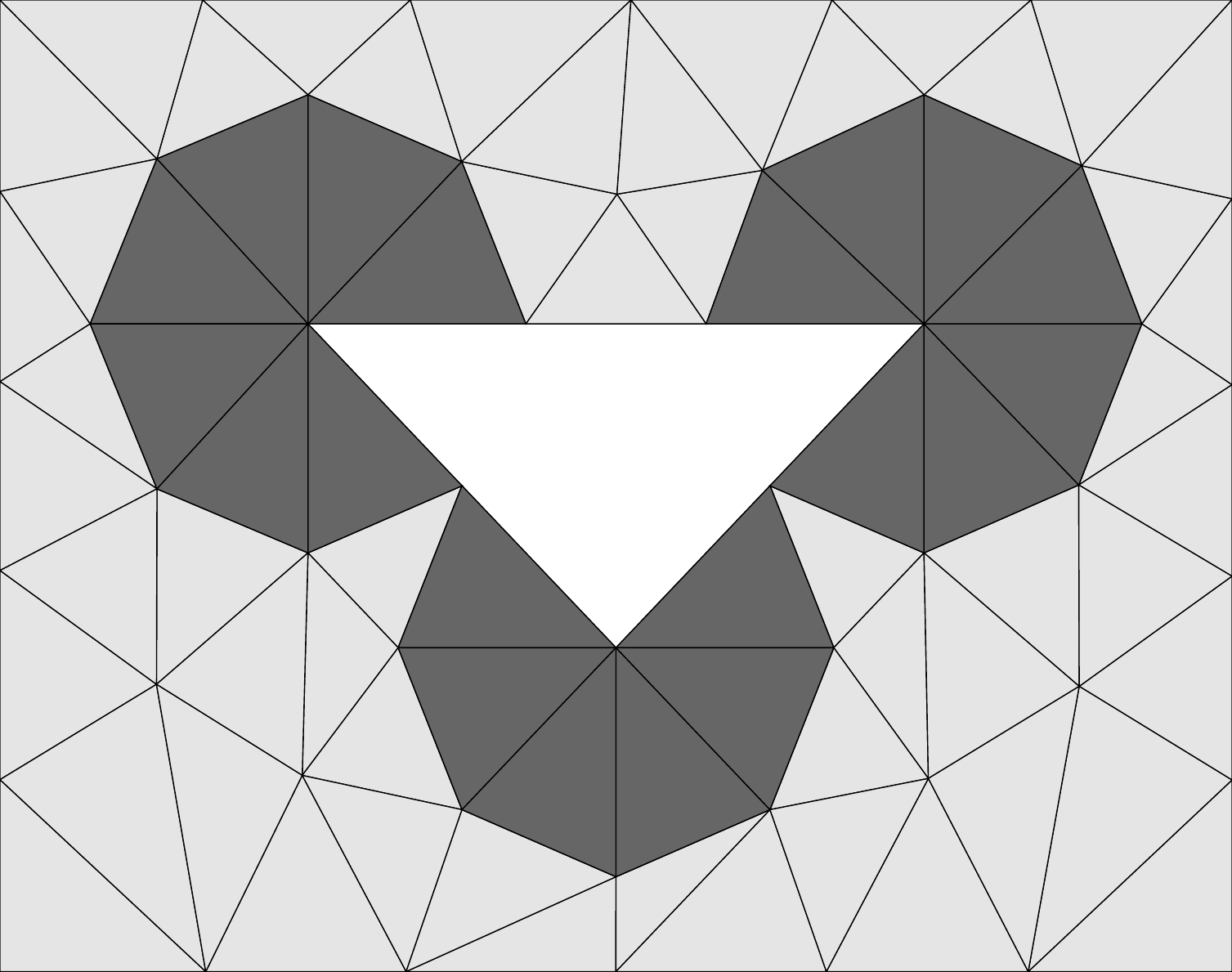}
}
\end{minipage}
\begin{minipage}{0.51\textwidth}
\huge
\resizebox{\textwidth}{!}{
\includegraphics{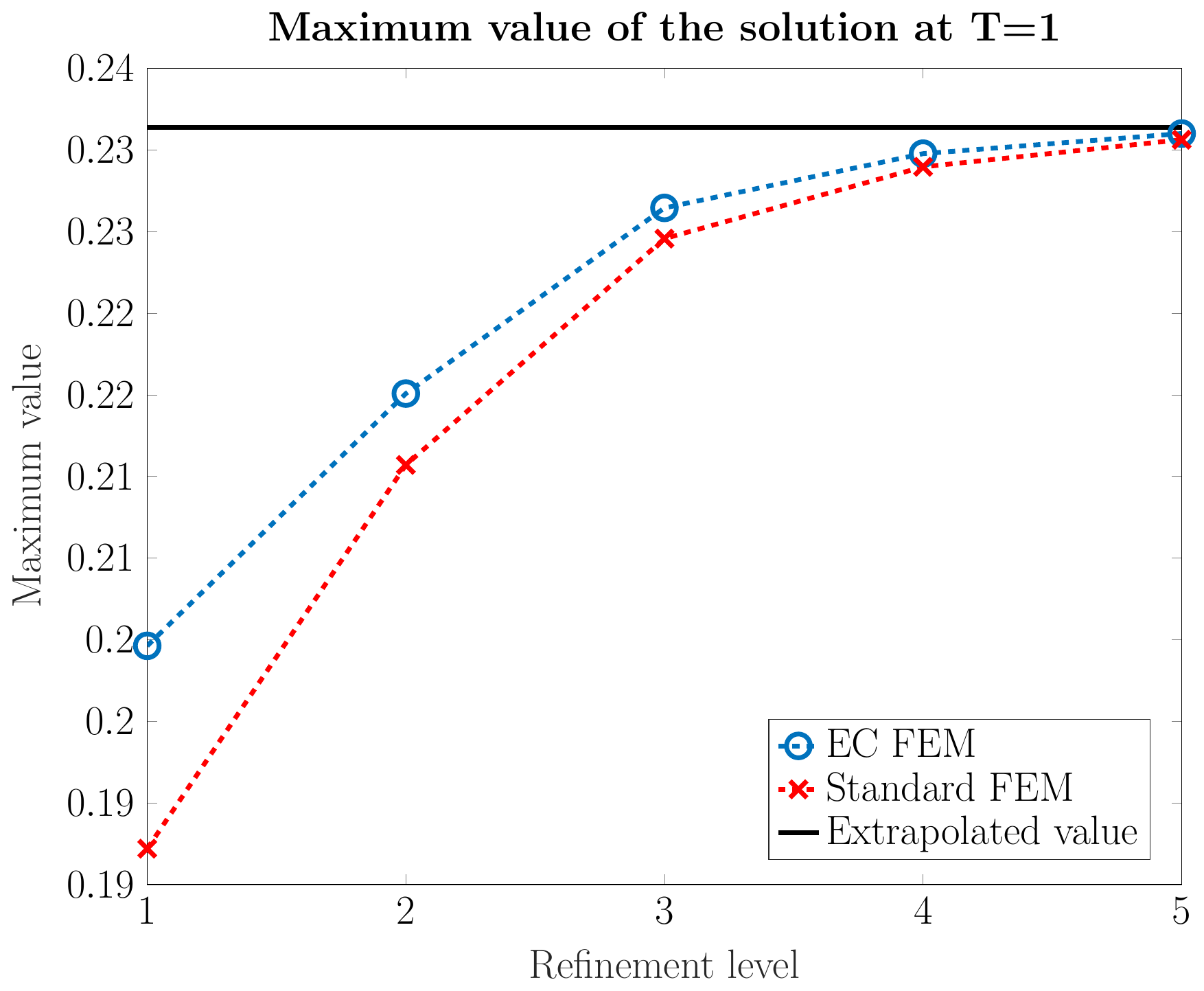}
}
\end{minipage} 
\caption{On the left-hand side a computational domain consisting of a rectangle $(0, 4)\times (0,3)$ with a right-angled, isosceles triangle cut out of it is shown. The domain is triangulated so that one element patches around the re-entrant corners consist of identical isoscles triangles. On the right-hand side a convergence of a computed quantity of interest - maximum value in the domain, is shown.}\label{fig:Triangle}
\end{figure}

\subsection{Piecewise polynomial FEM}\label{sec:HO}
In~\cite{Horger2016HO} ideas presented in Section~\ref{sec:EnergyCorrection} were extended to a more general setting of piecewise polynomial finite elements. Optimal convergence of the approximation~\eqref{eq:ECFiniteElementModified} of the model Poisson problem~\eqref{eq:EllipticModelProblem}, when using $k$-th order polynomial basis functions, is then obtained upon the choice of $f\in H^{k-1}_{-\alpha}(\Omega)$, and one of the modifying bilinear functions
\begin{align}\label{eq:ModificationHO}
c^R_h(u, v) := \sum_{i = 1}^K \gamma^R_i \int_{\omega^i_h} \nabla u \cdot \nabla v \; \mathrm{d} x, \quad c^F_h(u, v) := \sum_{i = 1}^K \gamma^F_i \int_{\omega^1_h} \hat{r}^{i-1}\nabla u \cdot \nabla v \; \mathrm{d} x.
\end{align}
Here, $\omega^i_h$ denotes the $i$-th layer of elements, counting from the considered corner and $\hat{r}$ is a distance from the corner measured on the reference triangle. We also assume that the patch of elements around the corner consists of identical isosceles triangles. 
Asymptotically, unique optimal sequences of parameters $\underline{\gamma}^{R*} = \big(\gamma^{R*}_i\big)_{i=1}^K $ and $\underline{\gamma}^{F*}=\big(\gamma^{F*}_i\big)_{i=1}^K$ on a given correction patch exist.

We now focus our attention on cubic basis functions, which means that for $f\in H^2_{-\alpha}(\Omega)$ we have
\begin{align*}
\| u - u_h^{\text{m}} \|_{2 + \alpha} \leq c h^4 \| f\|_{2, -\alpha}, \quad \text{and } \| \nabla ( u - u_h^{\text{m}}) \|_{2 + \alpha} \leq c h^3 \|f \|_{2, -\alpha}.
\end{align*}
We apply the method to the parabolic problem~\eqref{eq:HeatEquationWeakSolution} in a straightforward manner, suitably modifying the formulation~\eqref{eq:FEMParabolic}.

We are aiming to construct a fast, explicit solver and this means that, in order to guarantee the stability of the method, the CFL condition $\Delta t\sim h^2$ needs to be satisfied. The use of cubic finite element basis yields fourth-order convergence in weighted $L^2(\Omega)$-norm, and we would like to balance it with a second-order time-stepping scheme. Similarly to the piecewise linear case discussed in Section~\ref{sec:NumericalResults}, optimal balancing of the errors stemming from the time and space discretisations means that the CFL condition is automatically satisfied and makes the use of explicit time-stepping scheme feasible. We use the second-order Runge-Kutta scheme, also known as the Heun's method. The fully discrete scheme can be written as
\begin{align}\label{eq:HOwithRK2}
 \tilde{\mathbf{U}}^{\text{m}}_{n+1} &= \mathbf{U}^{\text{m}}_{n} +\Delta t \mathbf{M}^{-1} \Big[ \mathbf{F}^n - \mathbf{S}^{\text{m}} \mathbf{U}^{\text{m}}_{n} \Big]\\
\mathbf{U}^{\text{m}}_{n+1} &= \mathbf{U}^{\text{m}}_{n} +\frac{1}{2}\Delta t \mathbf{M}^{-1} \Big[ \big(\mathbf{F}^n - \mathbf{S}^{\text{m}} \mathbf{U}^{\text{m}}_{n}\big) + \big( \mathbf{F}^{n+1} - \mathbf{S}^{\text{m}} \tilde{\mathbf{U}}^{\text{m}}_{n+1} \big) \Big]
\end{align}
Note that the application of the mass-lumping strategy is not as straightforward as in the piecewise linear case, since in general it results in a singular mass matrix for the higher-order finite element. 
To overcome this, we follow the method proposed in~\cite{Cohen_2001} in the context of the wave equation. It is based on the enrichment of the cubic finite element space with three fourth-order polynomial bubble functions, which are uniformly equal to~$0$ at the elements' edges. Moreover, the mass matrix is assembled using a positive quadrature rule, which is exact for seventh-order polynomials, with quadrature points located in the nodal points of the enriched space.  Such a construction yields a diagonal matrix $\widetilde{\mathbf{M}}$ used in place of~$\mathbf{M}$ in Scheme~\eqref{eq:HOwithRK2}.

In Figure~\ref{fig:SchemesComparison}, we compare the accuracy and normalised computational times of the energy-corrected schemes with several commonly used methods for discretising the heat equation. We use a known exact solution $u = \sin(t) s_1 + \sin(2 t) s_2 - \sin(3t) s_3$ on the L-shaped domain $\Omega = (-1, 1)^2 \setminus \big( [0, 1]\times [-1, 0]\big)$ with the largest interior angle of size $\Theta = 3\pi/2$ and compare the $L^2(\Omega)$ and $L^2(\Omega')$ errors of the schemes at the last time step $T = 1$, where $\Omega' = \Omega \cap \{ |x | > 0.25\}$.
\begin{figure}
\centering
\begin{minipage}{0.49\textwidth}
\huge
\resizebox{\textwidth}{!}{
\includegraphics{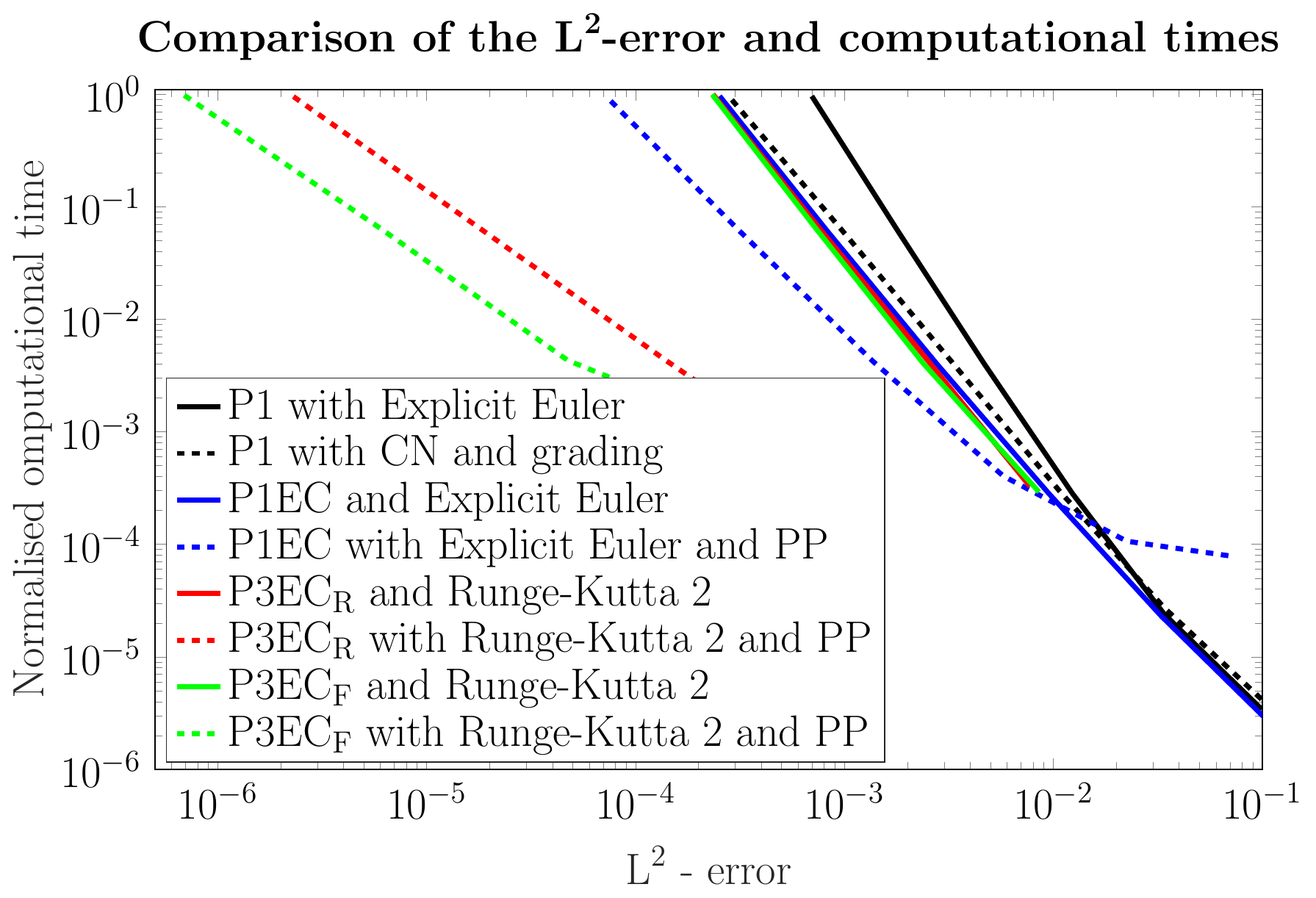}
}
\end{minipage}
\begin{minipage}{0.49\textwidth}
\huge
\resizebox{\textwidth}{!}{
\includegraphics{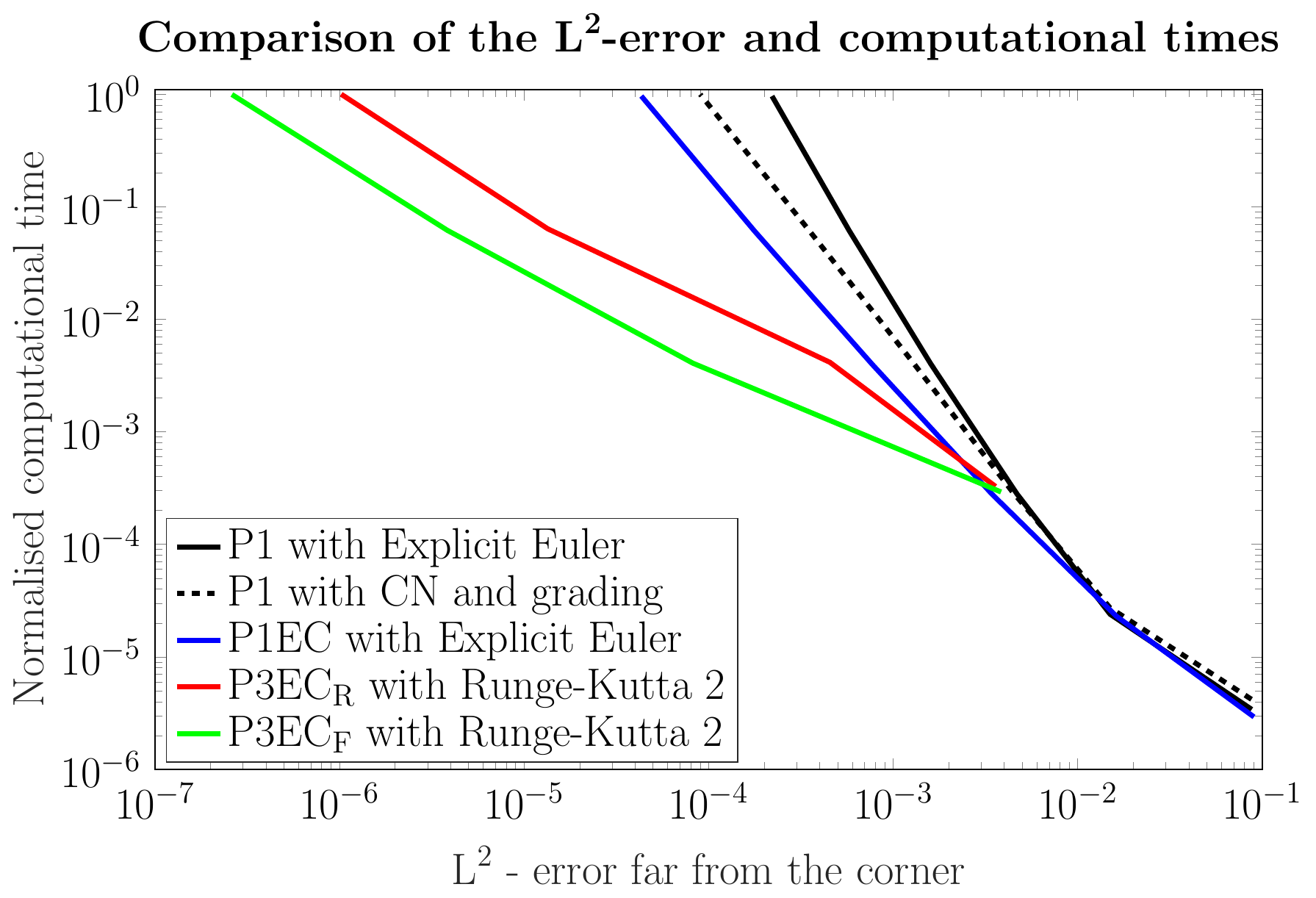}
}
\end{minipage} 
\caption{Comparison of the normalised computational time and accuracy of commonly used finite element discretisations of the parabolic problems with the proposed energy-corrected solvers on the L-shape domain. We compare the $L^2(\Omega)$ (left) and $L^2(\Omega')$ (right) error evaluated at the final time-step $T=1$.}\label{fig:SchemesComparison} 
\end{figure}

On the left-hand side of Figure~\ref{fig:SchemesComparison}, standard $L^2(\Omega)$ errors and normalised computational times are shown. The standard piecewise linear finite element combined with an Explicit Euler time-stepping and mass-lumping provides the worst results among the ones studied since its performance is limited by~\eqref{eq:ConvStFEMParabolic}. Initially $\Delta t = 0.1$ is chosen and with each space refinement the time step is divided by~$4$. 

Application of mesh grading techniques improves the performance of the solver~\cite{Chatzipantelidis2006}. However, the use of explicit time-stepping schemes is infeasible due to the very restrictive CFL condition. In order to recover the optimal convergence order in the $L^2(\Omega)$ norm, it is necessary to grade the mesh towards the singular corner, introducing elements of the size $h^{1/\mu}$, where $\mu < \pi / \Theta$. This in particular means that in the case of the L-shape domain with $\Theta = 3\pi / 2$, time-steps $\Delta t$ smaller than $\mathcal{O}(h^3)$ need to be used. Therefore, we use an unconditionally stable Crank-Nicolson scheme in time allowing for a coarser time discretisation. In order to keep the right balance between space and time discretisation errors, with each mesh refinement we divide the time-step by $2$, beginning with $\Delta t =0.1$.

Piecewise linear energy-corrected finite element scheme with Explicit Euler time-stepping~\eqref{eq:P1withEEDiscrete} yields significantly better results than the standard piecewise linear discretisations. It also gives comparable results with the mesh grading scheme completed with Crank-Nicolson time-stepping. An application of the post-processing additionally improves the accuracy of the method resulting in a better error-to-time ratio than the mesh grading method equipped with Crank-Nicolson time-stepping. We use in-built MATLAB linear system solvers. Note that the application of fast iterative solvers, such as multigrid methods, could additionally improve the performance of the relevant implicit methods.

Application of the cubic energy-corrected finite element scheme with second-order Runge-Kutta scheme in time gives similar results to the piecewise linear energy-corrected scheme. However, additional application of the post-processing yields the best results in terms of the balance between the computational time and the accuracy of the scheme out of all tested methods. This can be attributed to the use of the scheme eliminating the pollution effect in the solution, completion with the post-processing strategy yielding optimal convergence in the standard norms, and the use of mass-lumping strategy. Note that the use of $c_h^F(\cdot, \cdot)$ modification gives quantitatively better results than $c_h^R(\cdot, \cdot)$. This phenomenon was previously observed in~\cite[Section~6.3]{Horger2016HO} and can be attributed to the smaller modification subregion in the computational domain.

As shown in Section~\ref{Sec:FEM_parabolic}, energy-correction method gives optimal convergence rates in terms of the best-approximation property, however, in weighted norms. This, in particular, means that the method converges optimally when measured far from the re-entrant corner. Therefore, no additional post-processing needs to be applied, when one is interested in the solution far from the singular corner.

On the right-hand side of Figure~\ref{fig:SchemesComparison}, a comparison of $L^2(\Omega')$ errors and normalised computational times of the previously described methods are shown.	 Again, due to the pollution effect, the standard finite element discretisation results in the worst error-to-time ratio. It can be improved by the application of the mesh grading together with Crank-Nicolson time-stepping, which yields only slightly worse results than the piecewise-linear energy-corrected scheme. The cubic energy-corrected finite element, together with Heun's time-stepping and mass-lumping strategy, results in by far the best method when the $L^2$-error far from the re-entrant corner is concerned.
Small variations in the convergence rates in the cubic finite element scheme, when using $c^R_h(\cdot, \cdot)$ modification, appear because of insufficient initial resolution of the mesh. Again, modification $c^F_h(\cdot, \cdot)$ yields a better performance than $c^R_h(\cdot, \cdot)$.
\subsection{Application}
In this section, we apply the piecewise linear energy-corrected finite element method to a real 3D geometry of a graphite moderator brick of a nuclear power plant. Such a moderation type is commonly used in Advanced Gas-cooled Reactors (AGR)~\cite{Nonbol1996}. Efficient simulations of heat distribution in moderator bricks play an important role in the analysis of the material properties of the whole nuclear core, and accurate computations of temperature distribution can help determine the lifetime of nuclear materials, which often suffer from large temperature gradients and fast neutron fluxes~\cite{Arregui2016}.
\begin{figure}
    \centering
    \begin{subfigure}[b]{0.35\textwidth}
        \includegraphics[width=\textwidth]{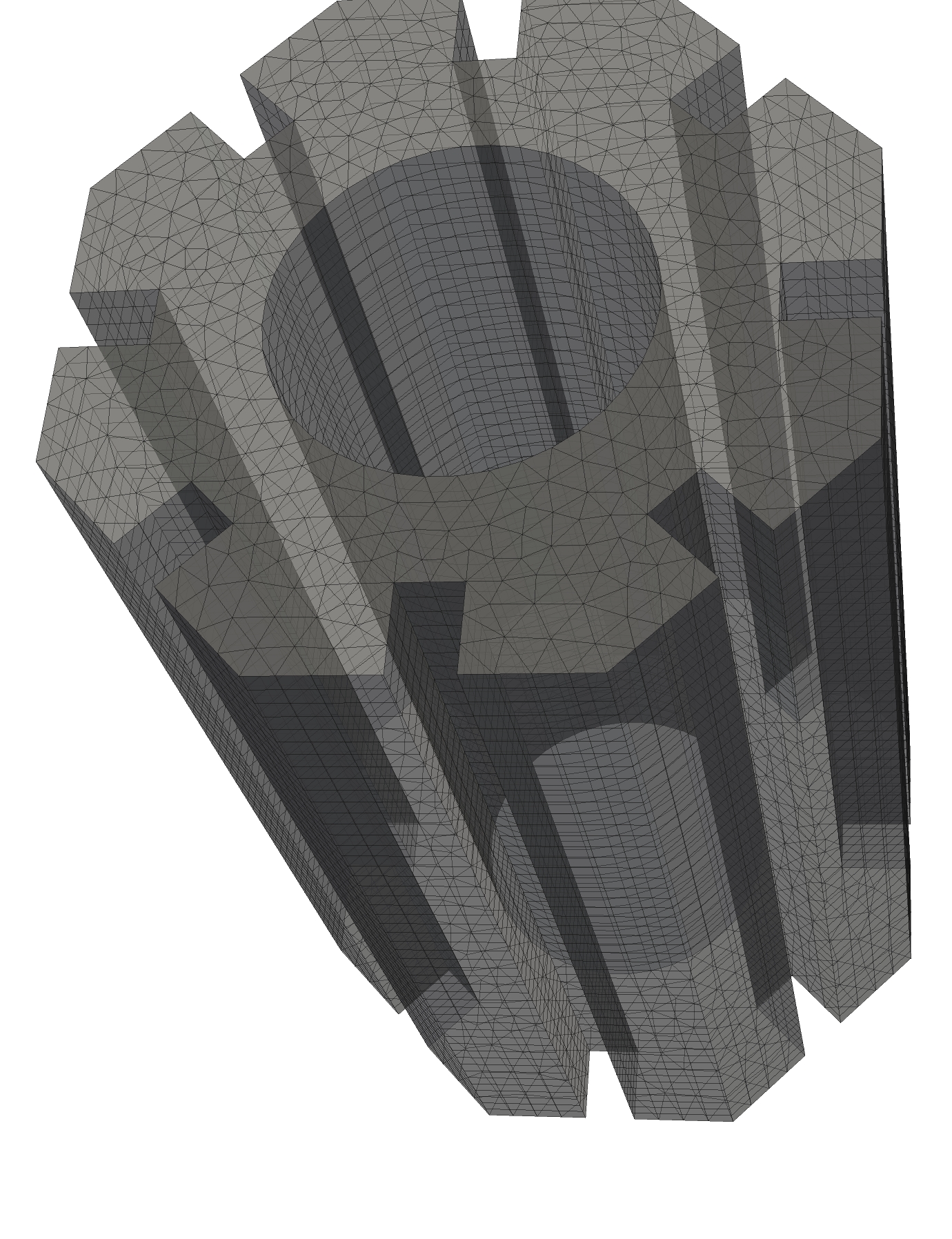}
        \caption{3D geometry}
        \label{fig:Geometry}
    \end{subfigure}
    ~ 
    \begin{subfigure}[b]{0.48\textwidth}
        \includegraphics[width=\textwidth]{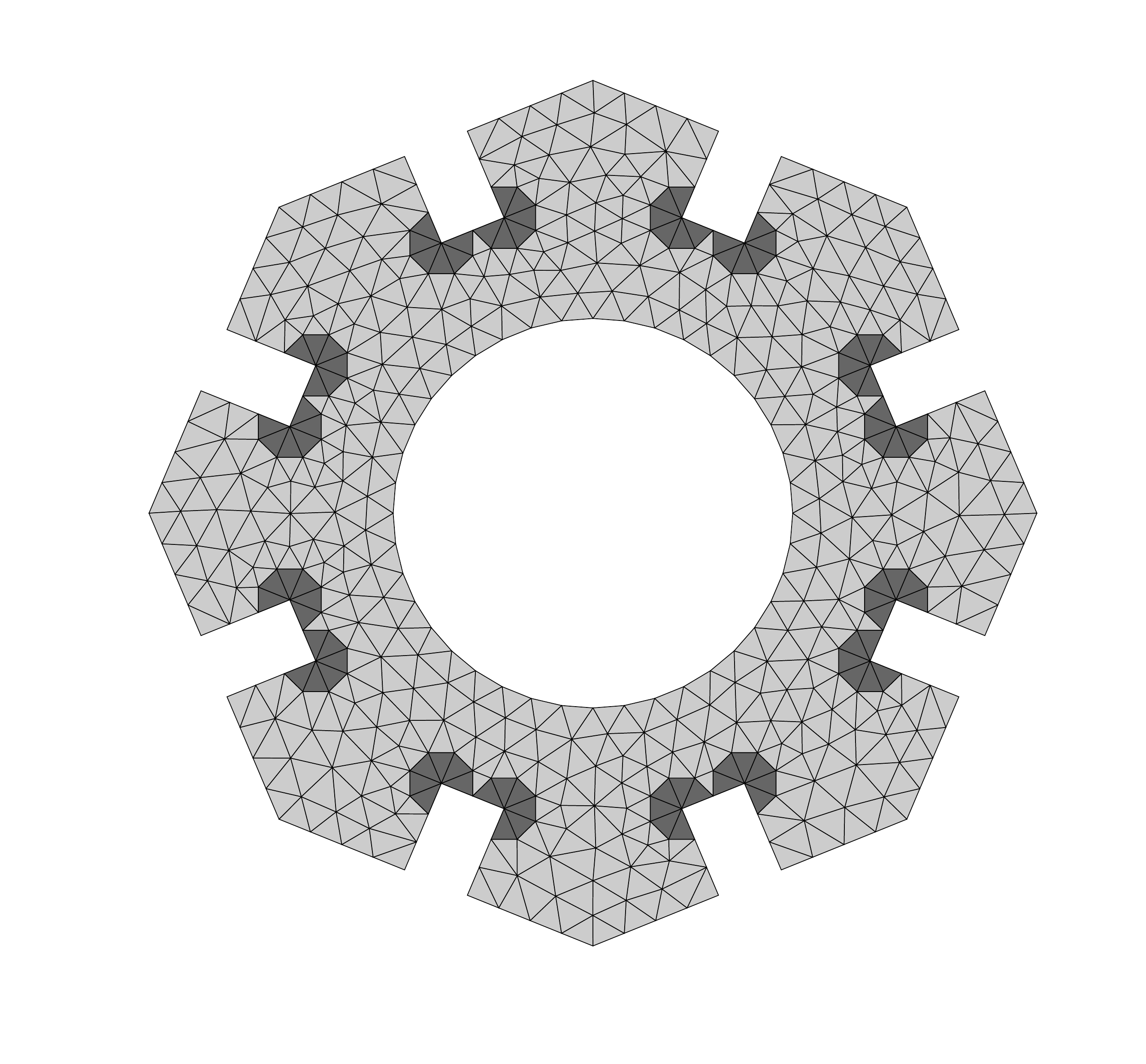}        
        \caption{Computational mesh of a cross-section}
        \label{fig:Mesh}
    \end{subfigure}
    \caption{Geometry and a computational mesh of a graphite moderator brick of a nuclear power plant. In Figure~\ref{fig:Geometry} a complete 3D domain is shown.}\label{fig:Fig1}
\end{figure}

The 3D geometry $\Omega$ of the graphite brick is presented in Figure~\ref{fig:Geometry}. The shape of the brick has a tensorial structure, with identical 2D horizontal cross-sections containing $16$ equally-sized re-entrant corners with angle~$\Theta = 3\pi/2$ at the external boundary. Heat transfer in a graphite moderator brick, in its simplest form, can be described by
\begin{align}\label{eq:HeatEquationNC}
u_t - \Delta u &= f \quad \text{in } \Omega\times (0, T),\\
u &= g \quad \text{on } \partial \Omega_1\times [0, T],\\
\partial_\nu u &= 0 \quad \text{on } \partial \Omega_2\times [0, T],\\
u &= 0 \quad \text{on } \partial \Omega_3\times [0, T],\\
u &= u_0 \quad \text{in } \Omega \text{ at } t = 0 .
\end{align}
Here, $\partial \Omega_1$ is the interior, cylindrical boundary of the domain and $g$ is the heating produced due to the nuclear reaction occurring in the fuel assembly. The system is thermally isolated from below on the $\partial \Omega_2$ part of the boundary, which is reflected by the uniform Neumann boundary conditions. Finally, the remaining part of the domain's boundary - $\partial\Omega_3$, is subject to a circulating coolant of a constant temperature. Note that the solution is rescaled, so that the temperature there is uniformly distributed.

Exploiting the tensorial structure of the domain $\Omega$, we divide it into prismatic elements of equal length $h_z$ in the vertical dimension. Moreover, each cross-section is triangulated as shown in Figure~\ref{fig:Mesh}. It is worth noting that around each of the $16$ re-entrant corners in the cross-sections, we use identical one-element patches consisting of congruent isosceles triangles. This, together with the tensorial structure of the mesh, allows us to reuse the parameter~$\gamma$ in the energy-corrected scheme~\eqref{eq:ECModification} once computed in the two-dimensional setting. Following~\eqref{eq:EETimeStepping}, we complete the finite element discretisation in space with the Explicit Euler time-stepping scheme.

\begin{figure}
    \centering
    \begin{subfigure}[b]{0.45\textwidth}
        \includegraphics[width=\textwidth]{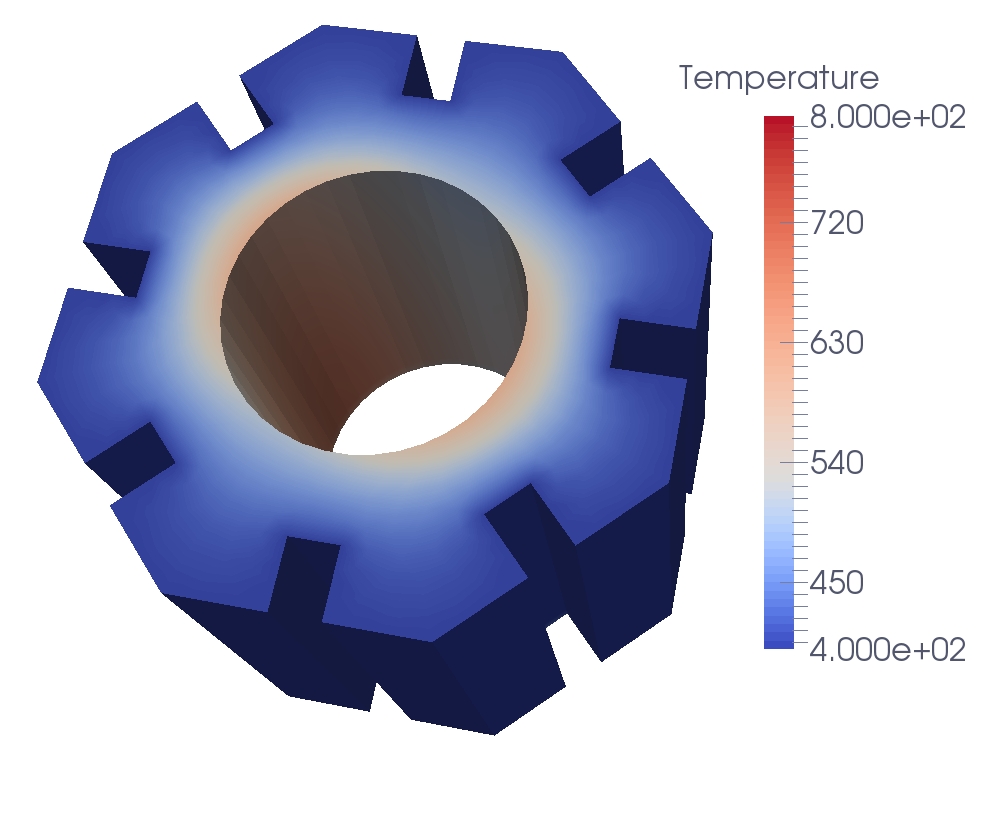}
        \caption{Temperature distribution at $t = 0.25$.}
        \label{fig:t025}
    \end{subfigure}
    ~ 
    \begin{subfigure}[b]{0.45\textwidth}
        \includegraphics[width=\textwidth]{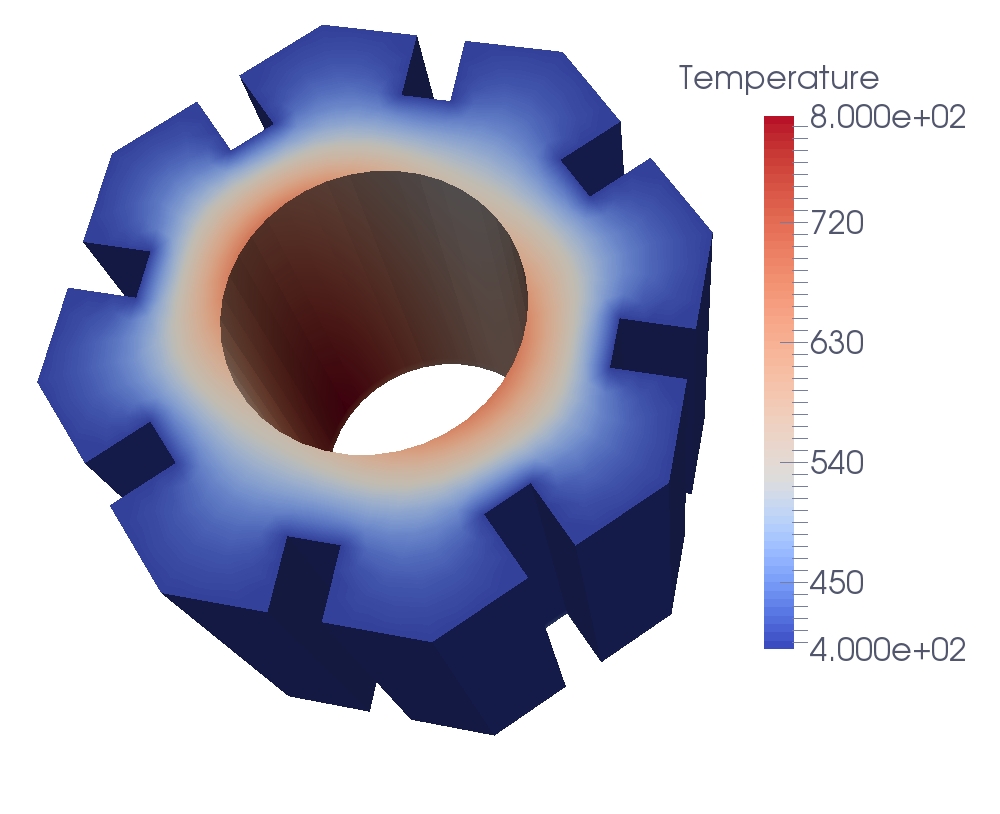}
        \caption{Temperature distribution at $t = 0.5$.}
        \label{fig:t05}
    \end{subfigure}
    \caption{Temperature distribution in the geometry of a graphite moderator brick at two different time points. The temperature is shown in a cross-section of a domain at the height $L/2 = 1.5$. We would like to point out the nonuniformity of the solution and lack of rotational symmetry due to choice of the boundary conditions on the internal, cylindrical wall.}\label{fig:Fig2}
\end{figure}
In the simulations, we choose a homogenuous initial temperature distribution $u_0 = 0$ and the heating on the internal boundary to be given by $g = 10 \cos \big( \pi z /(2 L) \big) \big( 1 + \sin(\phi + 4\pi t) \big) \sin(\pi t)$, where $L =3$ is the height of the domain $\Omega$ and $(r, \phi, z)$, $z\in [0, L]$ are cylindrical coordinates associated with the domain. The simulation is performed using $N= 5000$ time steps with the final time at $T=1$ and at the end of the simulation the solution is rescaled to a physically meaningful value range by $U = 20u + 400$. The solution at two intermediate time-steps is presented in the vertical cross-sections of the domain taken in the middle at the height $z =1.5$ at times $t = 0.25$ and $t = 0.5$.

To investigate the convergence of the scheme, we measure two different quantities of interest, namely the average temperature in the body at the final time-step and the average temperature in the whole space-time cylinder $\Omega\times [0, T]$
\begin{align*}
\text{QoI}_1 = \frac{1}{|\Omega|} \int_\Omega u_h(T, x) \; \mathrm{d}x, \quad \text{QoI}_2 = \frac{1}{T|\Omega|} \int_0^T \int_\Omega u_h(t, x) \; \mathrm{d}x\; \mathrm{d}t.
\end{align*}
In order to investigate the convergence properties of the quantities of interest, we perform the computations on four different refinement levels. Initially, we use the mesh described above, which is then uniformly refined and the time-step size is divided by~$4$. The results of the computations are presented in Figure~\ref{fig:Fig3}. We compare the accuracy of the standard finite element with the energy-corrected finite element method. Together with the quantities of interest $\text{QoI}_1$ and $\text{QoI}_2$, we include also the extrapolated values $\text{QoI}^{ex}_1$ and $\text{QoI}^{ex}_2$ in the plots. As presented in Figure~\ref{fig:Fig3}, the application of energy-correction improves the approximation properties of the quantities of interest. Moreover, in the case of the standard finite element approximation the respective estimated orders of convergence of $|\text{QoI}^{ex}_1 - \text{QoI}_1|$ and $|\text{QoI}^{ex}_2 - \text{QoI}_2|$ are $1.37$ and $1.23$. The application of the energy-correction improves these orders and yields estimated values of $2.1$ and $2.12$, respectively.  

\begin{figure}[h]
\begin{minipage}{0.48\textwidth}
\huge
\resizebox{\textwidth}{!}{
\includegraphics{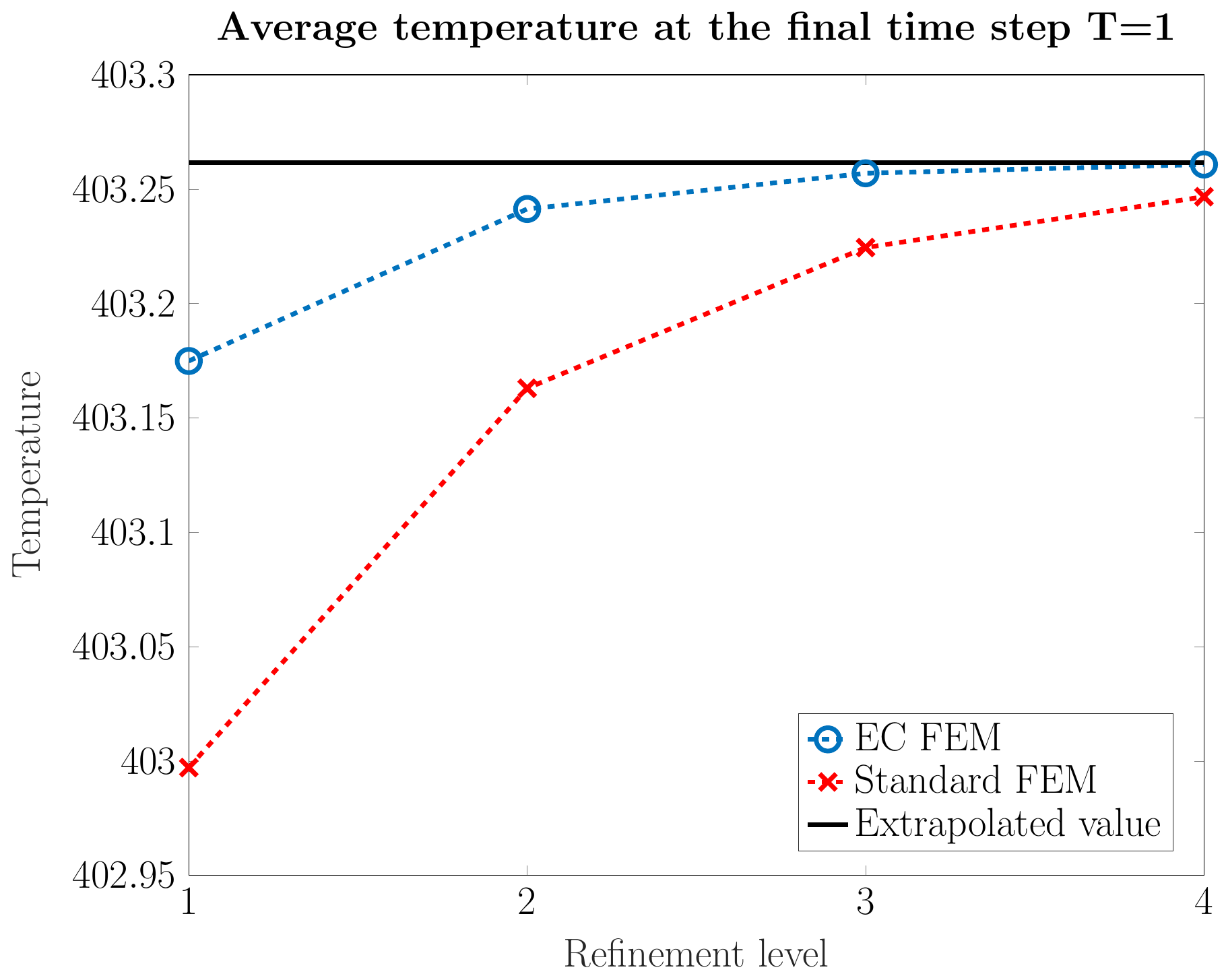}
}
\end{minipage}
\begin{minipage}{0.48\textwidth}
\huge
\resizebox{\textwidth}{!}{
\includegraphics{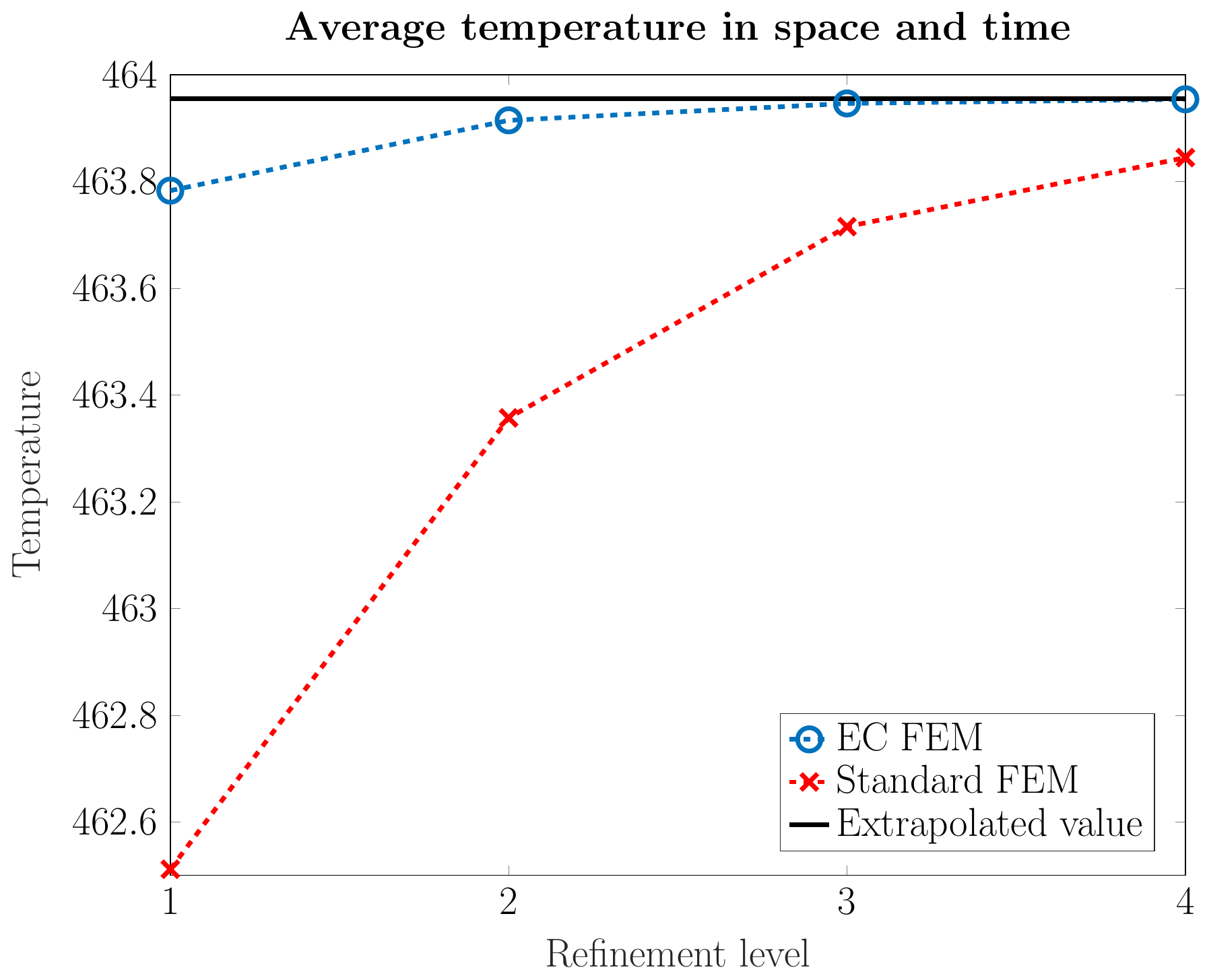}
}
\end{minipage} 
\caption{Comparison of the approximations of $\text{QoI}_1$ (left) and $\text{QoI}_2$ (right) obtained using standard and energy-corrected finite element scheme on four consecutive refinement levels in the space-time domain. Additionaly, extrapolated value of both quantities of interest are added to the plots for comparison.}\label{fig:Fig3}
\end{figure}
\section{Concluding remarks}
In this article, we proposed an energy-corrected finite element discretisation for parabolic problems on non-convex polygonal domains. We showed rigorously that the pollution effect inherited from elliptic problems can also be eliminated in the case of time-dependent problems, resulting in the optimal order of convergence in the sense of the best approximation property. The use of uniform meshes in the energy-correction method leads to a less restrictive CFL condition compared to standard methods based on mesh grading and adaptivity. Further, this allows for the application of explicit time-stepping schemes and creation of fast numerical schemes based on them.

Moreover, we proposed a post-processing approach and showed how the higher-order energy-corrected scheme can be combined with explicit Runge-Kutta type discretisation in time. This, together with the mass-lumping techniques, results in efficient solvers of parabolic problems for both piecewise linear and piecewise polynomial finite elements. Finally, we showed that the proposed algorithm can be successfully applied to 3D geometries with multiple re-entrant corners.

\section*{Acknowledgements}
We gratefully acknowledge the support of the German Research Fundation (DFG) through the grant WO 671/11-1 and, together with the Austrian Science Fund, through the IGDK1754 Training Group.

\bibliographystyle{plain}
\bibliography{library}
\end{document}